\documentclass[a4paper,11pt]{article}

\usepackage[top=3.2cm, bottom=3cm, left=3.45cm, right=3.45cm]{geometry}
\usepackage[T1]{fontenc}
\usepackage[utf8]{inputenc}
\usepackage{latexsym}
\usepackage{amsmath,amsthm,amssymb,amsfonts}
\usepackage{mathtools}
\usepackage[english]{babel}
\usepackage{tikz}
\usetikzlibrary{patterns, matrix, shapes}
\usepackage{tikz-cd}
\usepackage{booktabs}
\usepackage{soul} 
\DeclareMathOperator{\img}{Im}
\newcommand{\etal}{et~al.}

\mathchardef\mhyphen="2D
\def\loongmapsto#1{%
  \begin{tikzpicture}
    \draw (0,0.5mm) -- (0,-0.5mm);
    \draw[->] (0,0) -- (1.4, 0) node[above,midway] {\scriptsize #1};
  \end{tikzpicture}
}

\newcommand{\End}{\mathrm{End}}           
\newcommand{\Cay}{\mathrm{Cay}}           

\newcommand{\Ascseq}{\mathcal{A}}         
\newcommand{\Modasc}{\hat{\Ascseq}}       
\newcommand{\ntrees}{\mathrm{Tree}}       
\newcommand{\fubtrees}{\mathrm{Tree}^{\star}} 
\newcommand{\fishmat}{\mathcal{M}}        
\newcommand{\fishtrees}{\mathcal{T}}      
\newcommand{\fishpos}{\mathcal{P}}        
\newcommand{\twoplustwo}{\mathbf{2\hspace{-0.2em}+\hspace{-0.2em}2}}   

\newcommand{\asctops}{\mathrm{asctops}}   
\newcommand{\nub}{\mathrm{nub}}           
\newcommand{\pref}{\mathrm{pref}}         
\newcommand{\suff}{\mathrm{suff}}         

\newcommand{\treetops}{\mathrm{asctops}} 
\newcommand{\unseen}{\mathrm{nub}}       
\newcommand{\vertset}{\mathrm{V}}        
\newcommand{\lchild}{\mathrm{lchild}}     
\newcommand{\rchild}{\mathrm{rchild}}     
\newcommand{\lpath}{\mathrm{lpath}}       
\newcommand{\rpath}{\mathrm{rpath}}       
\newcommand{\diag}{\mathrm{diag}}         
\newcommand{\nondiag}{\overline{\diag}}   
\newcommand{\vlabel}{\mathfrak{l}}        
\newcommand{\blabel}{\mathfrak{b}}        

\newcommand{\level}{\mathrm{lev}}         

\newcommand{\treetoseq}{\alpha}           
\newcommand{\seqtotree}{\bar{\treetoseq}} 
\newcommand{\treetomat}{\beta}            
\newcommand{\mattotree}{\bar{\treetomat}} 
\newcommand{\treetopos}{\gamma}           
\newcommand{\postotree}{\bar{\treetopos}} 
\newcommand{\pairs}{\mathfrak{P}}         
\newcommand{\flip}{\mathrm{flip}}         
\newcommand{\sumop}{\mathrm{sum}}         

\newtheorem{theorem}{Theorem}[section]
\newtheorem{theorem*}{Theorem}[section]
\newtheorem{proposition}[theorem]{Proposition}
\newtheorem{lemma}[theorem]{Lemma}
\newtheorem{corollary}[theorem]{Corollary}

\newtheorem*{openproblem*}{Open Problem}

\theoremstyle{definition}
\newtheorem{definition}[theorem]{Definition}
\newtheorem{remark}[theorem]{Remark}
\newtheorem*{remark*}{Remark}
\newtheorem{example}{Example}
\newtheorem*{example*}{Example}

\title{Fishburn trees}
\author{Giulio Cerbai \and Anders Claesson} \date{}

\begin{document}
\thispagestyle{empty}

\maketitle
\begin{abstract}
  The in-order traversal provides a natural correspondence between
  binary trees with a decreasing vertex labeling and endofunctions on a
  finite set. By suitably restricting the vertex labeling we arrive at a
  class of trees that we call Fishburn trees. We give bijections between
  Fishburn trees and other well-known combinatorial structures that are
  counted by the Fishburn numbers, and by composing these new maps we
  obtain simplified versions of some of the known maps. Finally, we
  apply this new machinery to the so called flip and sum problems on
  modified ascent sequences.
\end{abstract}

\section{Introduction}
\thispagestyle{empty}

The coefficients of the elegant power series
$$
\sum_{n\geq 0}\prod_{k=1}^n\bigl(1-(1-x)^k\bigr)
\,=\, 1 +x + 2x^2 + 5x^3 + 15x^4 + 53x^5+\cdots
$$
are known as the Fishburn numbers, which is sequence A022493 in the
OEIS~\cite{Sl}. Claesson and Linusson~\cite{CL} named them so in honor
of Peter C.\ Fishburn (1936--2021), who pioneered, among other
things, the study of interval orders~\cite{F,F2}. The last decade has
seen a lot of interest in combinatorial structures related to this
counting sequence. The starting point was the 2010
paper~\cite{BMCDK} by Bousquet-M\'elou, Claesson,
Dukes and Kitaev, which gave one-to-one correspondences
between certain, apparently unrelated, objects:
$(\twoplustwo)$-free posets; the set of permutations avoiding a
certain bivincular pattern, now called Fishburn permutations;
Stoimenow matchings~\cite{St}; and ascent sequences.
They also provided an algorithm to transform an ascent sequence
into its modified version, and showed that the latter is
closely related to the level distribution of the corresponding
$(\twoplustwo)$-free poset.
Later, Dukes and Parviainen~\cite{DP} found a bijection
between ascent sequences and Fishburn matrices~\cite{F2}, that is,
upper triangular matrices with nonnegative integer entries whose
every row and column contains at least one positive entry.
All these objects are enumerated by the Fishburn numbers, and
for this reason we shall refer to them as \emph{Fishburn structures}.

We will define two new structures of this kind, namely \emph{Fishburn
  trees} and \emph{Fishburn covers}. The former are
decreasing binary trees satisfying some simple conditions on their
labeling, while the latter encode the trees as an ordered collection of
multisets. There are surprisingly straightforward bijections relating
them to modified ascent sequences, Fishburn matrices and
$(\twoplustwo)$-free posets. By composing these new maps we obtain
simplified versions of those previously known in the literature.  In
this sense, Fishburn trees and Fishburn covers provide a transparent
encoding of other Fishburn structures, and we may regard them as central
objects from which the others are derived. For instance, the Dukes and
Parviainen bijection~\cite{DP} is obtained by composing the map between
modified ascent sequences and Fishburn trees with the map between
Fishburn trees and Fishburn matrices.  As an application, we provide a
more direct solution to the $\flip$ and $\sumop$ problems (defined
below).

Our work fits into an active line of
research~\cite{CYZ,DJK,DKRS,DM,ES,FJLYZ,J,KR,L,Ya} that explores the relations
between Fishburn structures by analyzing how statistics and
operations that are natural on a certain object are
transported to the others. In this context, the following two
problems, originally proposed by Dukes and Parviainen~\cite{DP}, are
particularly relevant.

\begin{itemize}
\item \emph{The flip problem}. Duality acts as an involution on $(\twoplustwo)$-free
  posets. On Fishburn matrices, this is equivalent to reflecting a
  matrix in its antidiagonal. What is the corresponding operation on
  ascent sequences?
\item \emph{The sum problem}. The result of adding two Fishburn matrices
  is another Fishburn matrix. What is the corresponding operation on
  ascent sequences?
\end{itemize}

Note that the Dukes and Parviainen bijection between
ascent sequences and Fishburn matrices could
be used to compute the $\flip$ and
$\sumop$ operations. For instance, if $x$ is an ascent sequence, 
one could first determine the Fishburn matrix~$A$ corresponding
to~$x$, then compute $\flip(A)$ by reflecting $A$ in its 
antidiagonal, and finally go back to the desired ascent sequence
by applying the Dukes and Parviainen bijection once again.
This map is, however, defined by a rather intricate recursive
construction that makes this approach opaque. The goal is to find a
more transparent solution.
A first answer to the $\flip$ and $\sumop$ problems was proposed
by Ying and Yu~\cite{YY}. Roughly speaking, Ying and Yu
encode Fishburn matrices as, what they call, $M$-sequences, to then define a
bijection between ascent sequences and $M$-sequences
of Fishburn matrices. The $\flip$ and $\sumop$ are computed
on $M$-sequences, and the corresponding ascent sequences are
once again obtained by composition. Unfortunately, this solution is rather
cryptic, mainly due to the high
amount of technicalities, and the lack of a geometric description
of the construction. We believe that we have found a more transparent
construction; a key in making the construction more transparent is to view
it in terms of modified ascent sequences rather that plain ascent sequences.

In Section~\ref{section_prelim} we introduce a family of
decreasing binary trees called \emph{endotrees}.
We show that endotrees bijectively map to endofunctions via the in-order
traversal of the tree. To describe the inverse of this bijection, we define
the \emph{max-decomposition} of an endofunction and use it to recursively
build an endotree. Similarly, Cayley permutations are in
one-to-one correspondence with endotrees whose labels form an interval,
we call them \emph{regular endotrees}.

In Section~\ref{section_Fish_trees} we introduce Fishburn
trees as the set of regular endotrees satisfying an
additional property.
By decomposing a Fishburn tree in maximal right paths, each one
labeled with a unique integer, we are able to encode it
as an ordered collection of multisets, the \emph{Fishburn cover}.
The main result of this section, Theorem~\ref{pairs_to_tree}, is a
bijection from Fishburn covers to Fishburn trees.

In Section~\ref{section_trees_modasc}, we obtain a bijection
between Fishburn trees and modified ascent sequences by
restricting the in-order sequence and the max-decomposition.

In Section~\ref{section_trees_mat} we define a bijection
from Fishburn trees to Fishburn matrices by simply mapping
each maximal right path to a specific row of the matrix.
More specifically, we set the $(i,j)$-th entry of the matrix
equal to the number of nodes with label~$j$ contained
in the $i$-th path. On the other hand, a Fishburn matrix
naturally induces a Fishburn cover, and the corresponding
Fishburn tree is determined by Theorem~\ref{pairs_to_tree}.

The bijection from Fishburn trees to $(\twoplustwo)$-free
posets has a similar flavour, and is illustrated in
Section~\ref{section_trees_pos}.
Nodes with label~$i$ are mapped to the $i$-th level of the poset,
and the $j$-th strict down-set contains those that belong to
a maximal right path with index strictly less than~$j$.
Conversely, we show that Fishburn covers naturally define a
canonical labeling of $(\twoplustwo)$-free posets.

In Section~\ref{section_flipsum} we use Fishburn covers as stepping
stones to compute the $\flip$ and $\sumop$ of modified
ascent sequences.
We end this section with two concrete examples.

In Section~\ref{section_final} we provide a high-level
description of the framework introduced in this paper
and leave some open problems and suggestions for
future work.

\section{Endofunctions and decreasing binary trees}\label{section_prelim}

For any natural number $n$, let $\End_n$ be the set of
\emph{endofunctions}, $x:[n]\to[n]$, where
$[n]=\lbrace 1,2,\dots, n\rbrace$. We often identify an endofunction $x$
with the word $x=x_1\dots x_n$, where $x_i=x(i)$ for each $i\in[n]$.
Let $\End = \cup_{n\geq 0}\End_n$.  In general, given a set $A$ whose
elements are equipped with a notion of size, we will denote by $A_n$ the
set of elements in $A$ that have size $n$. Or, conversely, given a
definition of $A_n$ (of elements of size $n$) we let
$A=\cup_{n\geq 0}A_n$.  If $\img(x)=[k]$, for some $k\le n$, then $x$ is
a \emph{Cayley permutation}~\cite{Ca,MF}. The set of Cayley permutations is denoted
by $\Cay$. In other words, $x$ is a Cayley permutation if it contains at
least one copy of each integer between $1$ and its maximum element. For
example, $\Cay_1=\left\lbrace 1\right\rbrace$,
$\Cay_2=\left\lbrace 11,12,21\right\rbrace$ and
$$
  \Cay_3=\left\lbrace 111,112,121,122,123,132,211,212,213,221,231,312,321 \right\rbrace.
$$
There is a well known bijection between Cayley permutations and ballots
(ordered set partitions) of $[n]$. Indeed, a Cayley permutation $x$
encodes the ballot $B_1\dots B_k$, where $i\in B_{x(i)}$. In
particular, $|\Cay_n|$ is the $n$-th Fubini number, which is sequence
A000670 in the OEIS~\cite{Sl}.

A \emph{binary tree} is either the empty tree or a triple
$$
  T=(L, r, R),
$$
where $r$ is a node called the \emph{root} of $T$ and $L$ and
$R$ are binary trees called the \emph{left subtree} and the
\emph{right subtree} of $T$, respectively. Equivalently, a binary tree
is a rooted plane tree where each node has either $0$ children; $1$ child,
which can be either a left or right child; or $2$ children, namely a
left child and a right child.

Let $T$ be a binary tree. We denote by~$\vertset(T)$ the set of nodes
of~$T$. The \emph{size} of $T$ is the
cardinality of $\vertset(T)$. Now, suppose that $T$ is equipped with a
\emph{vertex labeling} $\vlabel:\vertset(T)\to \{1,2,\dots\}$ assigning
to each node $v\in \vertset(T)$ a positive integer label
$\vlabel(v)$. Then, assuming that $T$ is nonempty, we let
$$
  \max(T)=\max\bigl\lbrace\vlabel(v):v\in\vertset(T)\bigr\rbrace
$$
denote the largest value among the labels of $T$. For convenience we also
let $\max(T)=0$ when $T=\emptyset$ is the empty tree.

A \emph{decreasing binary tree} is a vertex-labeled binary tree $T$ such
that either $T$ is empty or $T=(L,r,R)$, where
$L$ and $R$ are decreasing binary trees and
$$\vlabel(r)\geq \max\bigl\{\max(L), \max(R)\bigr\}.
$$
We say that $T$ is \emph{strictly decreasing to the left} if it is empty
or $\vlabel(r)>\max(L)$ and $L$ and $R$ are strictly decreasing to the
left. Less formally, a vertex-labeled binary tree $T$ is
decreasing if on any path from the root to a leaf we encounter the
labels in weakly decreasing order. It is strictly decreasing to the left
if on any such path when we take a left turn we encounter a smaller
label.

\begin{definition}
  A decreasing binary tree $T$ of size $n$ is said to be an
  \emph{endotree} if it is strictly decreasing to the left and
  $\vlabel(v)\in [n]$ for each $v\in\vertset(T)$.  The last condition
  may be more compactly written $\img(\vlabel)\subseteq[n]$. If, in
  addition, $\img(\vlabel)=[k]$, for some $k\leq n$, then $T$ is said to
  be \emph{regular}. We denote by $\ntrees$ the set of
  endotrees and by $\fubtrees$ the set of regular endotrees.
\end{definition}
Of the four endotrees of size 2 there is exactly one which is not
regular, namely
\[
  \bigl(\emptyset, 2, (\emptyset, 2, \emptyset)\bigr) \;=\,
  \begin{tikzpicture}[scale = .5, semithick, baseline=-20pt, level distance=1.8cm, inner sep=2.3pt]
    \node {2}
    child {node at(1.5,0){2}}
    ;
  \end{tikzpicture}
\]
Of size 3 there are 13 regular endotrees and they are illustrated in
Figure~\ref{figure_fubtrees}.

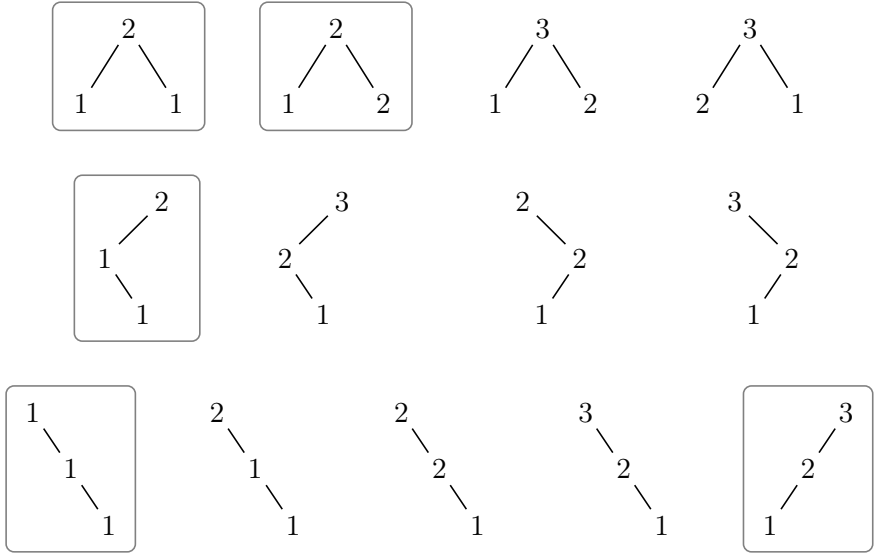
\begin{figure}
  \arraycolsep=10pt
  \def\arraystretch{2.0}
  \[
    \begin{array}{ccccc}
      \begin{tikzpicture}[scale = .5, semithick, baseline=0pt, level distance=2cm, inner sep=2.3pt]
        \draw[gray,rounded corners=3pt] (-2,-2.7) rectangle (2,0.7);
        \tikzstyle{level 1}=[sibling distance=2.5cm]
        \node {2}
        child {node {1}}
        child {node {1}}
        ;
      \end{tikzpicture}&
      \begin{tikzpicture}[scale = .5, semithick, baseline=0pt, level distance=2cm, inner sep=2.3pt]
        \draw[gray,rounded corners=3pt] (-2,-2.7) rectangle (2,0.7);
        \tikzstyle{level 1}=[sibling distance=2.5cm]
        \node {2}
        child {node {1}}
        child {node {2}}
        ;
      \end{tikzpicture}&
      \begin{tikzpicture}[scale = .5, semithick, baseline=0pt, level distance=2cm, inner sep=2.3pt]
        \draw[white,rounded corners=3pt] (-2,-2.7) rectangle (2,0.7);
        \tikzstyle{level 1}=[sibling distance=2.5cm]
        \node {3}
        child {node {1}}
        child {node {2}}
        ;
      \end{tikzpicture}&
      \begin{tikzpicture}[scale = .5, semithick, baseline=0pt, level distance=2cm, inner sep=2.3pt]
        \draw[white,rounded corners=3pt] (-2,-2.7) rectangle (2,0.7);
        \tikzstyle{level 1}=[sibling distance=2.5cm]
        \node {3}
        child {node {2}}
        child {node {1}}
        ;
      \end{tikzpicture}\\
    \end{array}
  \]
  \[
    \begin{array}{ccccc}
      \begin{tikzpicture}[scale = .5, semithick, baseline=0pt, level distance=1.5cm, inner sep=2.3pt]
        \draw[gray,rounded corners=3pt] (-2.3,-3.7) rectangle (1,0.7);
        \node {2}
        child {
            node at(-1.5,0){1}
            child {node at(1,0){1}}
          };
      \end{tikzpicture}&
      \begin{tikzpicture}[scale = .5, semithick, baseline=0pt, level distance=1.5cm, inner sep=2.3pt]
        \draw[white,rounded corners=3pt] (-2.3,-3.7) rectangle (1,0.7);
        \node {3}
        child {
            node at(-1.5,0){2}
            child {node at(1,0){1}}
          };
      \end{tikzpicture}&
      \begin{tikzpicture}[scale = .5, semithick, baseline=0pt, level distance=1.5cm, inner sep=2.3pt]
        \draw[white,rounded corners=3pt] (-2.3,-3.7) rectangle (1,0.7);
        \node {2}
        child {
            node at(1.5,0){2}
            child {node at(-1,0){1}}
          };
      \end{tikzpicture}&
      \begin{tikzpicture}[scale = .5, semithick, baseline=0pt, level distance=1.5cm, inner sep=2.3pt]
        \draw[white,rounded corners=3pt] (-2.3,-3.7) rectangle (1,0.7);
        \node {3}
        child {
            node at(1.5,0){2}
            child {node at(-1,0){1}}
          };
      \end{tikzpicture}
    \end{array}
  \]
  \[
    \begin{array}{ccccc}
      \begin{tikzpicture}[scale = .5, semithick, baseline=0pt, level distance=1.5cm, inner sep=2.3pt]
        \draw[gray,rounded corners=3pt] (-0.7,-3.7) rectangle (2.7,0.7);
        \node {1}
        child {
            node at(1,0){1}
            child {node at(1,0){1}}
          };
      \end{tikzpicture}&
      \begin{tikzpicture}[scale = .5, semithick, baseline=0pt, level distance=1.5cm, inner sep=2.3pt]
        \draw[white,rounded corners=3pt] (-0.7,-3.7) rectangle (2.7,0.7);
        \node {2}
        child {
            node at(1,0){1}
            child {node at(1,0){1}}
          };
      \end{tikzpicture}&
      \begin{tikzpicture}[scale = .5, semithick, baseline=0pt, level distance=1.5cm, inner sep=2.3pt]
        \draw[white,rounded corners=3pt] (-0.7,-3.7) rectangle (2.7,0.7);
        \node {2}
        child {
            node at(1,0){2}
            child {node at(1,0){1}}
          };
      \end{tikzpicture}&
      \begin{tikzpicture}[scale = .5, semithick, baseline=0pt, level distance=1.5cm, inner sep=2.3pt]
        \draw[white,rounded corners=3pt] (-0.7,-3.7) rectangle (2.7,0.7);
        \node {3}
        child {
            node at(1,0){2}
            child {node at(1,0){1}}
          };
      \end{tikzpicture}&
      \begin{tikzpicture}[scale = .5, semithick, baseline=0pt, level distance=1.5cm, inner sep=2.3pt]
        \draw[gray,rounded corners=3pt] (-2.7,-3.7) rectangle (0.7,0.7);
        \node {3}
        child {
            node at(-1,0){2}
            child {node at(-1,0){1}}
          };
      \end{tikzpicture}
   \end{array}
  \]
  \caption{Regular endotrees of size 3. Fishburn trees have been highlighted.}\label{figure_fubtrees}
\end{figure}

The \emph{in-order traversal} of a binary tree $T=(L,r,R)$ is performed
as follows: recursively traverse the left subtree $L$, visit the root
$r$, and recursively traverse the right subtree $R$. For the rest of
this paper, we will denote by $v_i$ the $i$-th visited node in the
in-order traversal of $T$. The \emph{in-order sequence} of a
vertex-labeled binary tree $T$ is defined by $\alpha(T)=x_1\dots x_n$,
in which $x_i=\vlabel(v_i)$. We can alternatively define $\alpha(T)$
recursively as follows. If $T$ is the empty tree,
then $\alpha(T)$ is the empty string. Otherwise, $T$ is nonempty and we
can write $T=(L,r,R)$. Then
$$
  \treetoseq(T)=\treetoseq(L)\vlabel(r)\treetoseq(R).
$$
It is easy to see that if $T$ is an endotree (of size $n$ and maximum $k$),
then $\treetoseq(T)$ is an endofunction (of size $n$ and maximum $k$).
That is, we have a map
$$
  \treetoseq:\ntrees_n\to\End_n.
$$
We wish to define the inverse map $\seqtotree$.
The \emph{max-decomposition} of a nonempty endofunction $x=x_1\dots x_n$ is
$$
  x\,=\,\pref(x)\,x_m\,\suff(x),
$$
where $\pref(x)=x_1\dots x_{m-1}$, $\suff(x)=x_{m+1}\dots x_{n}$ and
$m=\min\bigl(x^{-1}(\max(x))\bigr)$ is the index of the leftmost
occurrence of $\max(x)=\max\{x_i: i\in[n]\}$ in $x$. The tree
$\seqtotree(x)$ is then defined using recursion: If $x$ is the empty
word, then $\seqtotree(x)$ is the empty tree. Otherwise, $x$ is nonempty
and using the max-decomposition we can write $x=\pref(x)x_m\suff(x)$. Now, let
$$
  \seqtotree(x)=\bigl(L,r,R\bigr)
$$
be the tree with root $r$ labeled~$\vlabel(r)=x_m$, left subtree
$L=\seqtotree\bigl(\pref(x)\bigr)$ and right subtree
$R=\seqtotree\bigl(\suff(x)\bigr)$.

\begin{proposition}
  If $x\in\End_n$, then $\seqtotree(x)\in\ntrees_n$.
\end{proposition}
\begin{proof}
  Let $T=\seqtotree(x)$. If $x$ is empty, then there is nothing to
  prove. Assume that $x$ is nonempty and let $x=\pref(x)x_m\suff(x)$ be
  its max-decomposition. By definition of the map $\seqtotree$ we may
  write $T=(L,r,R)$, where $\vlabel(r)=x_m$,
  $L=\seqtotree\bigl(\pref(x)\bigr)$ and
  $R=\seqtotree\bigl(\suff(x)\bigr)$. Since $x_m$ is the leftmost
  occurrence of $\max(x)$ in $x$, each label in $L$ is strictly smaller
  than the label $x_m$ of the root $r$. Similarly, $T$ is weakly
  decreasing to the right since each label in $R$ is at most equal to
  $x_m$. The result follows from applying the induction hypothesis to $L$
  and $R$.
\end{proof}

\begin{proposition}
  The inverse map of $\treetoseq$ is $\seqtotree$.
\end{proposition}
\begin{proof}
  Using induction we shall show that $\treetoseq\circ\seqtotree$ is the
  identity function on $\End_n$, and that $\seqtotree\circ\treetoseq$
  is the identity function on $\ntrees_n$. The base cases are trivial
  and omitted. Assume $n\geq 1$. Let $x\in\End_n$. Applying the induction
  hypothesis to $\pref(x)$ and $\suff(x)$ we find that
  $$
    \treetoseq\bigl(\seqtotree(x)\bigr)
    = \treetoseq\bigl(\seqtotree(\pref(x)\bigr)x_m\treetoseq\bigl(\seqtotree(\suff(x)\bigr)
    = \pref(x)x_m\suff(x)
    = x.
  $$
  Let $T=(L,r,R)\in\ntrees_n$. By definition of $\treetoseq$, we have
  \begin{equation}\label{eq1}
    \treetoseq(T)=\treetoseq(L)\vlabel(r)\treetoseq(R).
  \end{equation}
  Since $T$ is strictly decreasing to the left, we have
  $\vlabel(r)>\vlabel(u)$ for each node $u$ in
  $L$. Thus, Equation~\ref{eq1} is the max-decomposition of
  $\treetoseq(T)$ and, by induction,
  \begin{equation*}
    \seqtotree\bigl(\treetoseq(T)\bigr)
    = \bigl(\seqtotree(\treetoseq(L)),r,\seqtotree(\treetoseq(R))\bigr)
    = (L,r,R)
    = T.\qedhere
  \end{equation*}
\end{proof}

A corollary of the previous result is that~$\treetoseq:\ntrees\to\End$
is a size-preserving bijection with inverse~$\seqtotree$. Furthermore,
it is easy to see that $T\in\fubtrees_n$ if and only if
$\treetoseq(T)\in\Cay_n$. That is, the (restricted) map
$\treetoseq:\fubtrees\to\Cay$ is a size-preserving bijection between
$\fubtrees$ and $\Cay$.

\begin{corollary}
  For each $n\ge 1$ we have
  $$
    |\ntrees_n|=|\End_n|\quad\text{and}\quad|\fubtrees_n|=|\Cay_n|.
  $$
\end{corollary}

See Figure~\ref{figure_fishburn_tree} for a concrete endotree
and its corresponding endofunction. In Figure~\ref{figure_non_fishburn}
(on the left) an example of a decreasing binary tree that fails to be an
endotree is given.

\section{Fishburn trees and Fishburn covers}\label{section_Fish_trees}

Throughout the preceding section we have denoted by $T=(L,r,T)$ a binary
tree with root $r$, left subtree $L$ and right subtree $R$. In the same
vein, given a node $v\in \vertset(T)$, let $T(v)$ denote the subtree of
$T$ consisting of $v$ together with all the descendants of $v$, and let
$L(v)$ and $R(v)$ denote the left and right subtrees of $T(v)$, so that
$T(v)=(L(v), v, R(v))$. Recall that $v_i$ denotes the $i$-th visited
node in the in-order traversal of $T$; in particular, $v_1$ is the first
visited node and $L(v_1)=\emptyset$. Assuming that $T$ is an
endotree and that $x=\treetoseq(T)$ is the corresponding
endofunction, then $x_{i-1}<x_i$ if and only if $L(v_i)$ is
nonempty. Such an $x_i$ is called an ascent top; by convention and for
convenience we will also include $x_1$ among the ascent tops. This
justifies us defining
$$
  \treetops(T)=\lbrace v_1\rbrace\cup\lbrace v_i: L(v_i)\neq\emptyset\rbrace
$$
as the set consisting of $v_1$ and nodes that have a left child.  We also
define
$$
  \unseen(T)=\lbrace v_j: \vlabel(v_i)\neq\vlabel(v_j)\ \text{for each } i<j\rbrace
$$
as the set of nodes $v_j$ whose label $\ell=\vlabel(v_j)$ is the
first occurrence of $\ell$ in the in-order sequence of $T$. As illustrated in
Figure~\ref{figure_fishburn_tree}, we can represent an endotree
so that labels of nodes in $\unseen(T)$ are the ``leftmost'' occurrences
among the labels of $T$. The in-order sequence $\treetoseq(T)$ is then
obtained by simply reading the labels of $T$ from left to right. With
this in mind, we say that $v_j\in\unseen(T)$ is the leftmost
occurrence of $\vlabel(v_j)$ in $T$.

\begin{figure}
\begin{center}
\begin{tikzpicture}[scale = .4, semithick, baseline=0pt]
\node(1)  at (0,10){\underline{1}};
\node(2)  at (1,8){1};
\node(3)  at (2,12){\underline{5}};
\node(4)  at (3,10){5};
\node(5)  at (4,6){1};
\node(6)  at (5,8){\underline{3}};
\node(7)  at (6,14){\underline{8}};
\node(8)  at (7,12){8};
\node(9)  at (8,8){5};
\node(10) at (9,6){5};
\node(11) at (10,0){1};
\node(12) at (11,2){\underline{2}};
\node(13) at (12,0){2};
\node(14) at (13,4){\underline{4}};
\node(15) at (14,2){3};
\node(16) at (15,10){\underline{7}};
\node(17) at (16,8){3};
\node(18) at (17,16){\underline{9}};
\node(19) at (18,12){2};
\node(20) at (19,14){\underline{6}};
\node(21) at (20,12){1};
\node at (-1.5,-2){$x=$};
\node(31) at (0,-2.1){\underline{1}};
\node(32) at (1,-2){1};
\node(33) at (2,-2.1){\underline{5}};
\node(34) at (3,-2){5};
\node(35) at (4,-2){1};
\node(36) at (5,-2.1){\underline{3}};
\node(37) at (6,-2.1){\underline{8}};
\node(38) at (7,-2){8};
\node(39) at (8,-2){5};
\node(40) at (9,-2){5};
\node(41) at (10,-2){1};
\node(42) at (11,-2.1){\underline{2}};
\node(43) at (12,-2){2};
\node(44) at (13,-2.1){\underline{4}};
\node(45) at (14,-2){3};
\node(46) at (15,-2.1){\underline{7}};
\node(47) at (16,-2){3};
\node(48) at (17,-2.1){\underline{9}};
\node(49) at (18,-2){2};
\node(50) at (19,-2.1){\underline{6}};
\node(51) at (20,-2){1};
\draw (2) -- (1) -- (3) -- (4) -- (6) -- (5);
\draw (3) -- (7) -- (8) -- (16) -- (9) -- (10) -- (14) -- (12)-- (11);
\draw (12) -- (13);
\draw (14) -- (15);
\draw (16) -- (17);
\draw (7) -- (18) -- (20) -- (19);
\draw (20) -- (21);
\draw[dashed, ultra thin, gray] (1) -- (31);
\draw[dashed, ultra thin, gray] (2) -- (32);
\draw[dashed, ultra thin, gray] (3) -- (33);
\draw[dashed, ultra thin, gray] (4) -- (34);
\draw[dashed, ultra thin, gray] (5) -- (35);
\draw[dashed, ultra thin, gray] (6) -- (36);
\draw[dashed, ultra thin, gray] (7) -- (37);
\draw[dashed, ultra thin, gray] (8) -- (38);
\draw[dashed, ultra thin, gray] (9) -- (39);
\draw[dashed, ultra thin, gray] (10) -- (40);
\draw[dashed, ultra thin, gray] (11) -- (41);
\draw[dashed, ultra thin, gray] (12) -- (42);
\draw[dashed, ultra thin, gray] (13) -- (43);
\draw[dashed, ultra thin, gray] (14) -- (44);
\draw[dashed, ultra thin, gray] (15) -- (45);
\draw[dashed, ultra thin, gray] (16) -- (46);
\draw[dashed, ultra thin, gray] (17) -- (47);
\draw[dashed, ultra thin, gray] (18) -- (48);
\draw[dashed, ultra thin, gray] (19) -- (49);
\draw[dashed, ultra thin, gray] (20) -- (50);
\draw[dashed, ultra thin, gray] (21) -- (51);
\end{tikzpicture}
\end{center}
  \caption{A Fishburn tree $T$ and the in-order sequence
    $x=\treetoseq(T)$. Nodes are spaced so that the
    in-order sequence is obtained by reading the labels of the nodes
    from left to right. Finally, labels of nodes in
    $\treetops(T)=\unseen(T)$, as well as the corresponding entries in
    $\asctops(x)=\nub(x)$, are underlined.}\label{figure_fishburn_tree}
\end{figure}
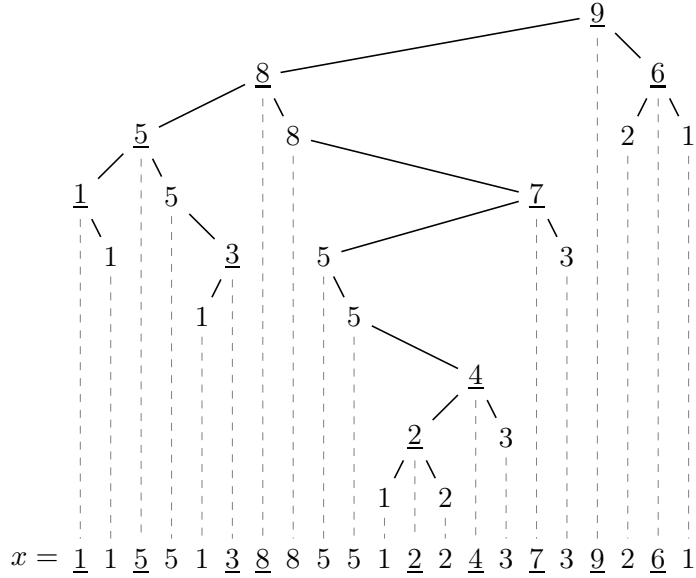

We are now ready to give the definition of Fishburn tree.

\begin{definition}
  A \emph{Fishburn tree} is a regular endotree~$T$ where $\unseen(T)=\treetops(T)$, and
  we denote by $\fishtrees$ the set of Fishburn trees.
\end{definition}

An example of a Fishburn tree is given in
Figure~\ref{figure_fishburn_tree}; an example of a non-Fishburn tree is
given in Figure~\ref{figure_non_fishburn} (on the right).  Five of the
13 endotrees of size 3 are Fishburn trees; they are highlighted in
Figure~\ref{figure_fubtrees}.  We continue this section with a couple of
simple lemmas concerning Fishburn trees.

\begin{lemma}\label{label_of_v1}
  If $T\in\fishtrees$, then $\vlabel(v_1)=1$.
\end{lemma}
\begin{proof}
  Let $v_j\in\unseen(T)$ be the leftmost occurrence of $1$ in $T$. By
  definition of Fishburn tree we have $\unseen(T)=\treetops(T)$ and thus
  $v_j\in\treetops(T)$. Hence the disjunction $j=1$ or
  $L(v_j)\neq\emptyset$ holds true. The latter disjunct is, however, false
  since $\vlabel(v_j)=1$ and $T$ is strictly decreasing to the left.
\end{proof}

\begin{lemma}\label{labels_interval}
  Let $T\in\fishtrees$ and let $k=\max(T)$. Then
  $$
    |\treetops(T)|=k\quad\text{and}\quad\vlabel\bigl(\treetops(T)\bigr)=[k].
  $$
\end{lemma}
\begin{proof}
  We have $\img(\vlabel)=[k]$. In particular, $\unseen(T)$ contains
  exactly one node with label $i$, for each $i\in [k]$. The claim then
  immediately follows from $\unseen(T)=\treetops(T)$.
\end{proof}

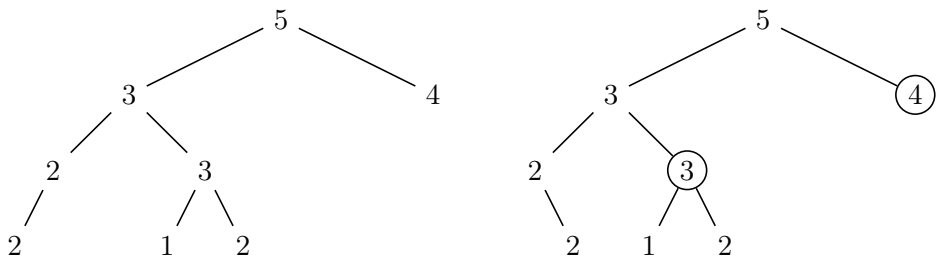
\begin{figure}
\begin{center}
\begin{tikzpicture}[scale = .5, semithick, baseline=0pt, level distance=2cm]
\tikzstyle{level 1}=[sibling distance=8cm]
\tikzstyle{level 2}=[sibling distance=4cm]
\tikzstyle{level 3}=[sibling distance=2cm]
\tikzstyle{level 4}=[sibling distance=1.5cm]
\node {5}
	child {node{3}
		child{node{2}
			child{node{2}}
			child[missing]{}			
			}
		child{node{3}
			child{node{1}}
			child{node{2}}
		}
	}
	child {node{4}}
;
\end{tikzpicture}
\qquad
\begin{tikzpicture}[scale = .5, semithick, baseline=0pt, level distance=2cm]
\tikzstyle{level 1}=[sibling distance=8cm]
\tikzstyle{level 2}=[sibling distance=4cm]
\tikzstyle{level 3}=[sibling distance=2cm]
\tikzstyle{level 4}=[sibling distance=1.5cm]
\node {5}
	child {node{3}
		child{node{2}
			child[missing]{}
			child{node{2}}			
			}
		child{node[draw, circle, inner sep=2pt]{3}
			child{node{1}}
			child{node{2}}
		}
	}
	child {node[draw, circle, inner sep=2pt]{4}}
;
\end{tikzpicture}\quad\mbox{}
\end{center}
\caption{The tree on the left is not an endotree since it is
not strictly decreasing to the left. The tree on the right is a regular
endotree, but not a Fishburn tree since $\unseen(T)\neq\treetops(T)$.
Indeed, if $u$ and $v$ are the distinguished nodes with $\vlabel(u)=3$ and
$\vlabel(v)=4$, respectively, then $u\in\treetops(T)\setminus\unseen(T)$
and $v\in\unseen(T)\setminus\treetops(T)$.
Note that the two trees have the same in-order sequence
$x=22313254$.}\label{figure_non_fishburn}
\end{figure}

A \emph{maximal right path} of a binary tree $T$ is a nonempty sequence of nodes
$W=(w_1,w_2,\dots,w_k)$ such that $w_{i+1}$ is the right child of $w_i$,
for each $i=1,\dots,k-1$; and $W$ is maximal in the sense that the first
node $w_1$ is not the right child of any node and the last node $w_k$
has no right child. \emph{Maximal left path} is defined analogously.
It is easy to see that for any node $v\in V(T)$ there is a unique
maximal right path to which $v$ belongs. Similarly, there is a unique
maximal left path to which $v$ belongs. We shall denote those by
$\rpath(v)$ and $\lpath(v)$, respectively. Furthermore, we define the
\emph{diagonal} of a nonempty binary tree $T=(L,r,R)$ by
$$
  \diag(T)=\lpath(r).
$$
Note that $\diag(T)\subseteq\treetops(T)$. We shall partition
$\treetops(T)$ accordingly as
$$
  \treetops(T)=\diag(T)\cup \nondiag(T),
$$
where $\nondiag(T)=\lbrace v\in\treetops(T): v\notin\diag(T)\rbrace$. We
shall also say that a node $v\in\treetops(T)$ is \emph{diagonal} if
$v\in\diag(T)$ and that it is \emph{non-diagonal} if $v\in\nondiag(T)$. Note
that $v_1$ is always diagonal.

Next we show that in a Fishburn tree the first node of a maximal right
path is either diagonal or the left child of a non-diagonal node.

\begin{lemma}\label{pathtop_partition}
  Let $W$ be a maximal right path of a Fishburn tree $T$ and let $w$ be
  the first node of $W$. Then either $w\in\diag(T)$ or $w$ is the left
  child of some $v\in\nondiag(T)$.
\end{lemma}
\begin{proof}
  Let $r$ be the root of $T$. Since $W$ is maximal, $w$ is not a right
  child. If $w$ is not a left child, then $w=r$ and thus
  $w\in\diag(T)$. Otherwise, $w$ is the left child of some
  $v\in\treetops(T)$. If $v\in\diag(T)$, then $w\in\diag(T)$ as well,
  since $\diag(T)=\lpath(r)$. Otherwise, $v\in\nondiag(T)$.
\end{proof}

The following corollary is an immediate consequence of
Lemma~\ref{pathtop_partition}.

\begin{corollary}\label{mrp_partition}
  For $T\in\fishtrees$ we have
  $$
    \vertset(T)=
    \!\bigcup_{v\in\diag(T)}\!\rpath(v)
    \;\;\cup\!
    \bigcup_{v\in\nondiag(T)}\!\rpath\bigl(\lchild(v)\bigr)
  $$
  where all the unions are disjoint and $\lchild(v)$\!
  denotes the left child of $v$.
\end{corollary}

We refer to the partition of $\vertset(T)$ induced by its maximal right
paths as the \emph{$\rpath$-decomposition} of $T$. Furthermore, we say that a
maximal right path $W$ is \emph{diagonal} if its first node is diagonal;
otherwise, if the first node is the left child of a non-diagonal node,
$W$ is \emph{non-diagonal}.

By Corollary~\ref{mrp_partition} each maximal right path $W$ of $T$ can
be associated with a node in $\treetops(T)$ in the following manner:
\begin{itemize}
\item If $W$ is diagonal, then it is associated with its first node,
  which is a diagonal node.
\item If $W$ is non-diagonal, then it is associated with the father
  of its first node, which is a non-diagonal node.
\end{itemize}

Conversely, each node in $\treetops(T)$ determines a unique maximal
right path this way. The correspondence between maximal right paths of
$T$ and $\treetops(T)$ described above is thus bijective. Now, recall
from Lemma~\ref{labels_interval} that $|\treetops(T)|=k$ and
$\vlabel\bigl(\treetops(T)\bigr)=[k]$, where $k=\max(T)$. In particular,
there are exactly $k$ maximal right paths in $T$. Moreover, if $W$ is a
maximal right path and $v$ is its associated node in $\treetops(T)$, then
we can assign the integer $\vlabel(v)\in[k]$ to it. Let $W_i$ be the
maximal right path assigned to $i\in[k]$ in this manner.
In particular, $W_k=\rpath(r)$, where $r$ is the root of $T$.
As an example, the maximal right paths of the Fishburn tree in
Figure~\ref{figure_fishburn_tree} are illustrated in
Figure~\ref{figure_ellipses}. Here below, to encode them more compactly,
we abuse notation and write down the corresponding labels:
\[
  \begin{array}{lll}
    W_1=(1,1) & W_4=(2,2)   & W_7=(5,5,4,3) \\
    W_2=(1)   & W_5=(5,5,3) & W_8=(8,8,7,3) \\
    W_3=(1)   & W_6=(2)     & W_9=(9,6,1).  \\
  \end{array}
\]

\begin{definition}\label{def_of_b}
  For each node $u\in\vertset(T)$ we let $\blabel(u)$ be the index of
  the maximal right path that contains $u$; i.e.\ $u\in W_{\blabel(u)}$.
\end{definition}

\begin{figure}
\begin{center}
\begin{tikzpicture}[scale = .4, semithick, baseline=0pt]
\node(1)  at (0,10){\underline{1}};
\node(2)  at (1,8){1};
\node(3)  at (2,12){\underline{5}};
\node(4)  at (3,10){5};
\node(5)  at (4,6){1};
\node(6)  at (5,8){\underline{3}};
\node(7)  at (6,14){\underline{8}};
\node(8)  at (7,12){8};
\node(9)  at (8,8){5};
\node(10) at (9,6){5};
\node(11) at (10,0){1};
\node(12) at (11,2){\underline{2}};
\node(13) at (12,0){2};
\node(14) at (13,4){\underline{4}};
\node(15) at (14,2){3};
\node(16) at (15,10){\underline{7}};
\node(17) at (16,8){3};
\node(18) at (17,16){\underline{9}};
\node(19) at (18,12){2};
\node(20) at (19,14){\underline{6}};
\node(21) at (20,12){1};
\draw (2) -- (1) -- (3) -- (4) -- (6) -- (5);
\draw (3) -- (7) -- (8) -- (16) -- (9) -- (10) -- (14) -- (12)-- (11);
\draw (12) -- (13);
\draw (14) -- (15);
\draw (16) -- (17);
\draw (7) -- (18) -- (20) -- (19);
\draw (20) -- (21);
\node[scale=1.25] at (-1.75,10.75){$W_{1}$};
\draw[fill=gray, opacity=0.15] (0.5,9) circle [x radius=0.75, y radius=2, rotate=30];
\node[scale=1.25] at (0.5,14){$W_{5}$};
\draw[fill=gray, opacity=0.15] (3.35,10) circle [x radius=0.75, y radius=4, rotate=35];
\node[scale=1.25] at (3,5){$W_{3}$};
\draw[fill=gray, opacity=0.15] (4,6) circle [x radius=0.5, y radius=1, rotate=30];
\node[scale=1.25] at (5.5,16){$W_{8}$};
\draw[fill=gray, opacity=0.15] (11,11) circle [x radius=2, y radius=7, rotate=60];
\node[scale=1.25] at (15,4){$W_{7}$};
\draw[fill=gray, opacity=0.15] (11,5) circle [x radius=1.25, y radius=5.5, rotate=45];
\node[scale=1.25] at (9.25,2.75){$W_{4}$};
\draw[fill=gray, opacity=0.15] (11.5,1) circle [x radius=0.75, y radius=2, rotate=30];
\node[scale=1.25] at (8.25,0.5){$W_{2}$};
\draw[fill=gray, opacity=0.15] (10,0) circle [x radius=0.5, y radius=1, rotate=30];
\node[scale=1.25] at (15,17){$W_{9}$};
\draw[fill=gray, opacity=0.15] (18.65,14) circle [x radius=0.75, y radius=4, rotate=35];
\node[scale=1.25] at (16.25,13){$W_{6}$};
\draw[fill=gray, opacity=0.15] (18,12) circle [x radius=0.5, y radius=1, rotate=30];
\end{tikzpicture}
\end{center}
\caption{The $\rpath$-decomposition of the Fishburn tree of 
Figure~\ref{figure_fishburn_tree}. Here $\diag(T)$ contains
the underlined nodes with label $1$, $5$, $8$ and $9$.
Thus the paths $W_1$, $W_5$, $W_8$ and $W_9$ are diagonal,
while $W_2$, $W_3$, $W_4$, $W_6$ and $W_7$ are non-diagonal.
}\label{figure_ellipses}
\end{figure}
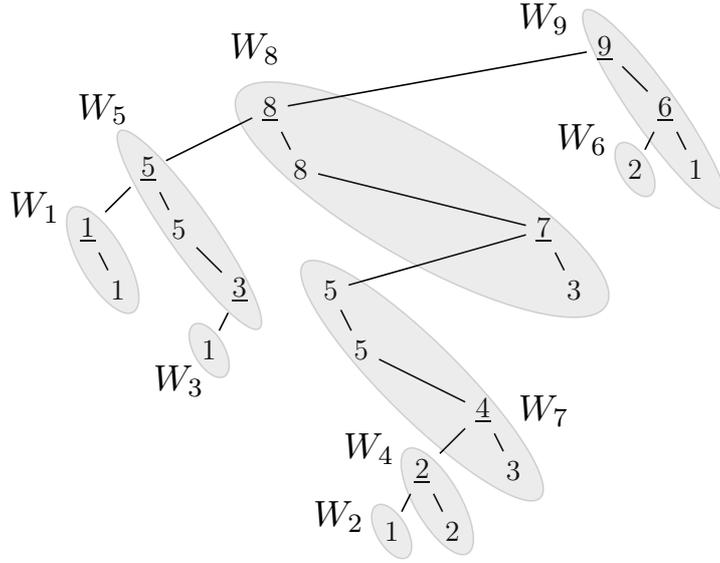

The label $\blabel(u)$ can be recursively computed as follows. The label
of the root $r$ of $T$ is $\blabel(r)=\vlabel(r)$, and for
$v\in\vertset(T)$, $v\neq r$, we have
\begin{equation}\label{eqn:recursive_def_of_b}
  \blabel(v) =
  \begin{cases}
    \blabel(u) & \text{if $v=\rchild(u)$}, \\
    \vlabel(v) & \text{if $v=\lchild(u)$ and $u\in\diag(T)$}, \\
    \vlabel(u) & \text{if $v=\lchild(u)$ and $u\in\nondiag(T)$}.
  \end{cases}
\end{equation}
Here, $\lchild(u)$ denotes the left child of $u$, and $\rchild(u)$
denotes the right child of $u$.  See Figure~\ref{figure_rules_blabel}
for an illustration of these rules.

\begin{figure}
  $$
  \begin{array}{c|c|c}
    \begin{tikzpicture}[scale = .5, semithick, baseline=0pt, level distance=2cm]
      \node {$u$} child {node at (1.5,0){$v$}};
      \node at (0.5,-3.5){$\blabel(v)=\blabel(u)$};
    \end{tikzpicture}
    \quad
    &
    \quad
    \begin{tikzpicture}[scale = .5, semithick, baseline=0pt, level distance=2cm]
      \node[white] at (0,0) {$u$} child {node at (-1.5,0){$v$}};
      \node at (1.76,0.035) {$u \in\diag(T)$};
      \node at (1,-3.5){$\blabel(v)=\vlabel(v)$};
    \end{tikzpicture}
    \quad
    &
    \quad
    \begin{tikzpicture}[scale = .5, semithick, baseline=0pt, level distance=2cm]
      \node[white] at (0,0) {$u$} child {node at (-1.5,0){$v$}};
      \node at (1.76,0.089) {$u \in\nondiag(T)$};
      \node at (1,-3.5){$\blabel(v)=\vlabel(u)$};
    \end{tikzpicture}
  \end{array}
  $$
  \caption{Rules to determine the label $\blabel(v)$.}\label{figure_rules_blabel}
\end{figure}

Let $B_i$ be a multiset containing a copy
of the integer $j$ for each node with label $j$ in $W_i$. That is, $B_i$
is the multiset $\lbrace\vlabel(u): u\in W_i\rbrace$.
Given a Fishburn tree $T$, we denote by $\pairs(T)$ the
list of multisets
$$
  \pairs(T)=B_1\dots B_k
$$
defined this way. Note that $\bigcup_{i\in [k]} B_i=[k]$, since
$\vertset(T)=[k]$. Furthermore, $j\le i$ for each $j\in B_i$,
since $T$ is decreasing.

\begin{definition}
  An ordered collection of $k$ nonempty multisets $P=B_1\dots B_k$ is
  a \emph{Fishburn cover} if the following two conditions are satisfied:
  \begin{itemize}
  \item $\displaystyle{\bigcup_{i\in [k]} B_i=[k]}$;
  \item for all $i\in [k]$, if $j\in B_i$, then $j\le i$.
  \end{itemize}
\end{definition}

As noted above, the $\rpath$-decomposition of a Fishburn tree $T$
determines a Fishburn cover $\pairs(T)$. For instance, the Fishburn
cover of the tree in Figure~\ref{figure_fishburn_tree} is
$$
  \pairs(T)=\lbrace 1,1 \rbrace\lbrace 1\rbrace\lbrace 1\rbrace
  \lbrace 2,2\rbrace\lbrace 5,5,3\rbrace\lbrace 2\rbrace
  \lbrace 5,5,4,3\rbrace\lbrace 8,8,7,3\rbrace\lbrace 9,6,1\rbrace,
$$
where, for reasons that will become clear later, the elements of a block
are written in weakly decreasing order. In Theorem~\ref{pairs_to_tree}, below, we show
that the converse is true as well; that is, every Fishburn cover
uniquely determines the $\rpath$-decomposition of a Fishburn tree. First
a simple lemma.

\begin{lemma}\label{diag_nondiag_paths}
  Let $T$ be a Fishburn tree and let $\pairs(T)=B_1B_2\dots B_k$ be the
  Fishburn cover of $T$. Then, for each $i\in [k]$,
  $$
    i\in B_i\;\iff\; \text{$W_i$ is diagonal}.
  $$
\end{lemma}

\begin{proof}
  The maximal right path $W_i$ is diagonal if and only if $i$ is the
  label of the first node of $W_i$, that is $i\in B_i$.
\end{proof}

\begin{figure}
\begin{tabular}{cc}
$\begin{tikzpicture}[scale = .45, semithick, baseline=0pt, level distance=2cm]
        \tikzstyle{level 1}=[sibling distance=8cm]
        \tikzstyle{level 2}=[sibling distance=4cm]
        \tikzstyle{level 3}=[sibling distance=2cm]
        \tikzstyle{level 4}=[sibling distance=1.5cm]
        \node {7}
        child {node{5}
            child{node{2}
                child{node{1}}
                child{node{1}}
              }
            child{node{4}
                child[missing]{}
                child{node{2}}
              }
          }
        child {node{6}
            child[missing]{}
            child{node{3}}
          }
        ;
        \node at(0,-9.5)[text width=6.25cm]{Step 0: The
          comb-shaped tree $T_0$ arising from the diagonal blocks
          $B_1=\lbrace 1\rbrace$, $B_2=\lbrace 2,1\rbrace$,
          $B_5=\lbrace 5,4,2\rbrace$ and $B_7=\lbrace 7,6,3\rbrace$.};
      \end{tikzpicture}$
&
$\begin{tikzpicture}[scale = .45, semithick, baseline=0pt, level distance=2cm]
        \tikzstyle{level 1}=[sibling distance=8cm]
        \tikzstyle{level 2}=[sibling distance=4cm]
        \tikzstyle{level 3}=[sibling distance=2cm]
        \tikzstyle{level 4}=[sibling distance=1.5cm]
        \node {7}
        child {node{5}
            child{node{2}
                child{node{1}}
                child{node{1}}
              }
            child{node{4}
                child[missing]{}
                child{node{2}}
              }
          }
        child {node{6}
            child{node{\underline{5}}
                child[missing]{}
                child{node{\underline{3}}
                    child[missing]{}
                    child{node{\underline{2}}}
                  }
              }
            child{node{3}}
          }
        ;
        \node at(0,-10)[text width=6.25cm]{Step 1: $T_1$ is
          obtained by attaching $W_6$ (i.e.\ $B_6=\lbrace 5,3,2\rbrace$)
          to $T_0$.};
      \end{tikzpicture}$
\\
$\begin{tikzpicture}[scale = .45, semithick, baseline=0pt, level distance=2cm]
        \tikzstyle{level 1}=[sibling distance=8cm]
        \tikzstyle{level 2}=[sibling distance=4cm]
        \tikzstyle{level 3}=[sibling distance=2cm]
        \tikzstyle{level 4}=[sibling distance=1.5cm]
        \node {7}
        child {node{5}
            child{node{2}
                child{node{1}}
                child{node{1}}
              }
            child{node{4}
                child{node{\underline{2}}
                    child[missing]{}
                    child{node{\underline{1}}}
                  }
                child{node{2}}
              }
          }
        child {node{6}
            child{node{5}
                child[missing]{}
                child{node{3}
                    child[missing]{}
                    child{node{2}}
                  }
              }
            child{node{3}}
          }
        ;
        \node at(0,-10)[text width=6cm]{Step 2: $T_2$ is
          obtained by attaching $W_4$ (i.e.\ $B_4=\lbrace 2,1\rbrace$) to
          $T_1$.};
      \end{tikzpicture}$
&
$\begin{tikzpicture}[scale = .45, semithick, baseline=0pt, level distance=2cm]
        \tikzstyle{level 1}=[sibling distance=8cm]
        \tikzstyle{level 2}=[sibling distance=4cm]
        \tikzstyle{level 3}=[sibling distance=2cm]
        \tikzstyle{level 4}=[sibling distance=1.5cm]
        \node {7}
        child {node{5}
            child{node{2}
                child{node{1}}
                child{node{1}}
              }
            child{node{4}
                child{node{2}
                    child[missing]{}
                    child{node{1}}
                  }
                child{node{2}}
              }
          }
        child {node{6}
            child{node{5}
                child[missing]{}
                child{node{3}
                    child{node{\underline{2}}}
                    child{node{2}}
                  }
              }
            child{node{3}}
          }
        ;
        \node at(0,-10)[text width=6cm]{Step 3: $T=T_3$ is
        obtained by attaching $W_3$ (i.e.\ $B_3=\lbrace 2\rbrace$)
        to $T_2$.};
      \end{tikzpicture}$
\end{tabular}
  \caption{The step-by-step construction of Theorem~\ref{pairs_to_tree}
    on the Fishburn cover
    $\pairs(T)=\lbrace 1\rbrace\lbrace 2,1\rbrace\lbrace 2\rbrace\lbrace
    2,1\rbrace\lbrace 5,4,2\rbrace\lbrace 5,3,2\rbrace \lbrace
    7,6,3\rbrace$. At each step, nodes in the newly appended path are
    underlined. Observe that, at the last step, $W_3$ is attached to the
    leftmost occurrence of $3$ in $T_2$. This is an example where the
    ordering in which the paths are attached matters. Indeed, if we had
    started by appending $W_3$ to $T_0$, then the first node of $W_3$
    would have been     attached to a different node (the only other node
    labeled~$3$) and the result would not have been a Fishburn tree.
    }\label{figure_comb_tree}
\end{figure}

\begin{theorem}\label{pairs_to_tree}
  For each Fishburn cover $P$, there is a unique Fishburn tree $T$ such
  that $\pairs(T)=P$.
\end{theorem}
\begin{proof}
  Let $P=B_1\dots B_k$ be a Fishburn cover.  We will construct, in
  multiple steps, a Fishburn tree $T$ such that $\pairs(T)=P$. For each
  $i,j\in [k]$, let $m_i(j)$ be the multiplicity of $j$ in $B_i$.
  Construct a decreasing binary tree $W_i$ consisting of a single right
  path with $|B_i|$ nodes in total and $m_i(j)$ nodes labeled $j$. It is
  easy to see that $\pairs(T)=P$ if and only if the
  $\rpath$-decomposition of $T$ is given by the paths $W_1$, \dots,
  $W_k$.

  Let $D=\lbrace i\in [k]: i\in B_i\rbrace$ and
  $\bar{D}=\lbrace i\in [k]: i\notin B_i\rbrace$.  Due to
  Lemma~\ref{diag_nondiag_paths}, we want to construct our tree in such
  a way that $W_i$ is diagonal if $i\in D$, and non-diagonal if
  $i\in\bar{D}$. We start by arranging the diagonal paths
  $\{W_i: i\in D\}$ in a comb-shaped decreasing binary tree tree
  $T_0$ (see Figure~\ref{figure_comb_tree}) so that
  $\diag(T_0)=\lbrace w_i: i\in D\rbrace$, where $w_i$ is the first node
  of $W_i$.

  To attach the remaining non-diagonal paths, $\{W_i: i\in \bar{D}\}$,
  we will specify an iterative procedure. Suppose that we are in $s$-th
  step of this procedure and that we have already constructed a tree
  $T_{s-1}$. Due to the equality $\unseen(T)=\treetops(T)$ defining
  Fishburn trees, we shall attach $W_i$, with $i\in\bar{D}$, to
  $T_{s-1}$ so that $w_i$ becomes the left child of the leftmost
  occurrence of $i$. To make sure that the procedure is well-defined and
  that the desired property is preserved, we start with the largest
  index in $\bar{D}$ and proceed in decreasing order. Once again, we
  refer to Figure~\ref{figure_comb_tree} for a step-by-step illustration
  of this construction.  Assume that
  $\bar{D}=\lbrace j_1,j_2,\dots,j_m\rbrace$ with
  $j_1>j_2>\dots>j_m$. For $s=1,2,\dots,m$:
  \begin{itemize}
  \item Let $y_s$ be the leftmost occurrence of $j_s$ in $T_{s-1}$.
  \item Let $T_s$ be the tree obtained by attaching the path $W_{j_s}$
    to $T_{s-1}$ so that $w_{j_s}$---the first node of
    $W_{j_s}$---becomes the left child of $y_s$.
  \end{itemize}
  For the succession of trees $T_0$, $T_1$, \dots, $T_m$ to be
  well-defined we need to verify that, for $s=1,2,\dots,m$,
  \begin{enumerate}
  \item The tree $T_{s-1}$ contains at least one node with label $j_s$.
  \item The node $y_s$, whose label is the leftmost occurrence of $j_s$
    in $T_{s-1}$, has no left child.
  \end{enumerate}
  To prove the first property, note that, since $P$ is a Fishburn cover,
  we have
  $$
    [k]=\bigcup_{i\in [k]}B_i=\bigcup_{i\in [k]}\vlabel(W_i).
  $$
  Thus at least one path, say $W_t$, contains a node with label
  $j_s$. If $t\in D$, then $W_t$ is contained in $T_0$. On the other
  hand, suppose that $t=j_q\in\bar{D}$, for some $q$. Note that $q<s$,
  or else $j_q\le j_s$ and $\vlabel(u)<j_q\le j_s$ for each
  $u\in W_{j_q}$, contradicting the assumption that $W_t$ contains
  a node with label $j_s$. Hence $W_{j_q}$ is contained in
  $T_q$, with $q<s$. In both cases $T_{s-1}$ contains at least one node
  with label $j_s$ and hence $y_s$ is well-defined.

  To prove the second property, note that
  $$
    \treetops(T_{s-1})=\lbrace w_i:i\in D\rbrace\cup\lbrace y_1,\dots, y_{s-1}\rbrace
  $$
  and no node in $\treetops(T_{s-1})$ has label $j_s$.

  Thus the succession of trees $T_0$, $T_1$, \dots, $T_m$ is
  well-defined and we let $T=T_m$. It remains to show that $T$ is a
  Fishburn tree and that it is the only Fishburn tree that satisfies
  $\pairs(T)=P$. The proof is divided into four parts corresponding
  to following claims, in which $s\in \{0,1,\dots,m\}$:
  \begin{enumerate}
  \item $\nondiag(T_s)=\lbrace y_1,\dots, y_s\rbrace$;
  \item $\treetops(T_s)\subseteq\unseen(T_s)$;
  \item $T=T_m$ is a Fishburn tree;
  \item $T$ is the only Fishburn tree such that $\pairs(T)=P$.
  \end{enumerate}

  \textit{Proof of Claim 1.} Note that $\nondiag(T_0)=\emptyset$.
  Assume that $s\in [m]$ and
  $\nondiag(T_{s-1})=\lbrace y_1,\dots, y_{s-1}\rbrace$.
  Now, $T_s$ is obtained from $T_{s-1}$ by attaching $W_{j_s}$ to
  $T_{s-1}$ so that $w_{j_s}$ becomes the left child of $y_s$.
  Hence,
  $$
    \nondiag(T_s)=\nondiag(T_{s-1})\cup\lbrace y_s \rbrace=
    \lbrace y_1,\dots, y_{s-1},y_s\rbrace.
  $$

  \textit{Proof of Claim 2.} Note that
  $\treetops(T_0)=\diag(T_0)\subseteq\unseen(T_0)$.
  Assume that $s\in [m]$ and
  $\treetops(T_{s-1})\subseteq\unseen(T_{s-1})$. Now,
  $$
    \treetops(T_s)=\treetops(T_{s-1})\cup\lbrace y_s\rbrace.
  $$
  Observe that $y_s\in\unseen(T_s)$ by construction. Furthermore, each
  node $u$ in the newly attached path $W_{j_s}$ has label
  $\vlabel(u)<j_s$ due to Lemma~\ref{diag_nondiag_paths} and
  the definition of Fishburn cover. Thus,
  $\vlabel(u)<j_s<j_q$ for each $q<s$ and hence
  $\lbrace y_1,\dots,y_{s-1}\rbrace\subseteq\unseen(T_s)$. Therefore,
  $\nondiag(T_s)=\lbrace y_1,\dots,y_{s-1},y_s\rbrace\subseteq\unseen(T_s)
  $.
  To obtain the desired inclusion $\treetops(T_s)\subseteq\unseen(T_s)$
  it suffices to prove that $\diag(T_s)\subseteq\unseen(T_s)$. Consider
  the (only) path $Q$ from the root of $T_s$ to $y_s$ and let $i$ be the
  label of the last diagonal top contained in $Q$ (see
  Figure~\ref{figure_claimII}).
  Let $v\in\diag(T_s)$. Recall that $v\in\unseen(T_{s-1})$. If
  $\vlabel(v)\le i$, then $v$ precedes each node of $W_{j_s}$ in the
  in-order traversal of $T_s$ and thus $v\in\unseen(T_s)$. On the other
  hand, if $\vlabel(v)>i$, then $\vlabel(u)<i<\vlabel(v)$ for each
  $u\in W_{j_s}$, hence we have $v\in\unseen(T_s)$ once again.

  \textit{Proof of Claim 3.} Note that $T$ is decreasing and strictly
  decreasing to the left by construction. Moreover, due to what we proved
  above, we have
  $$
    \treetops(T)\subseteq\unseen(T)\quad\text{and}\quad
    |\treetops(T)|=k.
  $$
  Consequently, $|\unseen(T)|=k$ as well, from which the equality
  $\treetops(T)=\unseen(T)$ follows and hence $T$ is a Fishburn tree.
  
  \textit{Proof of Claim 4.} Let $T'$ be a Fishburn tree with
  $\pairs(T')=P$. We will show that $T'=T$. Since $\pairs(T')=\pairs(T)$,
  the $\rpath$-decomposition of $T'$ is given by the same paths
  $W_1,\dots,W_k$. In particular, in $T'$ the diagonal paths
  $\lbrace W_i: i\in D\rbrace$ must be arranged in a comb-shaped tree
  $T'_0$ such that $T'_0=T_0$. We wish to prove that each of the
  remaining paths $\lbrace W_i: i\in\bar{D}\rbrace$ is attached to the
  same node as in $T$. That is, for $s=1,\dots,m$, the path $W_{j_s}$
  is attached to the leftmost occurrence $y_s$ of $j_s$ in $T_{s-1}$,
  where $\bar{D}=\lbrace j_1,\dots,j_m\rbrace$ and $j_1>j_2>\cdots>j_m$.
  Consider the path $W_{j_1}$. Note that $\vlabel(u)<j_1$ for each
  $u\in W_{j_t}$ and $t\ge 1$, hence there are no nodes with label $j_1$
  in the paths $W_{j_1},\dots,W_{j_m}$. Therefore, $y_1$ is the leftmost
  occurrence of $j_1$ not only in $T'_0=T_0$, but also in every tree obtained
  by attaching $W_{j_1},\dots,W_{j_m}$ to $T'_0$. In particular, due to the
  usual equality $\treetops(T')=\unseen(T')$ defining Fishburn trees,
  $W_{j_1}$ must be attached to $y_1$ in $T'$. In other words, the
  subtree $T'_1$ of $T'$ consisting of $T'_0$ and the path $W_{j_1}$
  is $T'_1=T_1$.
  Similarly, we have $\vlabel(u)<j_2$ for each $u\in W_{j_t}$ and
  $t\ge 2$. Thus $y_2$ is the leftmost occurrence of $j_2$ in every
  tree obtained by attaching $W_{j_2},\dots,W_{j_m}$ to $T'_1$,
  and $W_{j_2}$ must be attached to $y_2$ in $T'$. The remaining
  paths can be addressed analogously.
  \end{proof}
  
  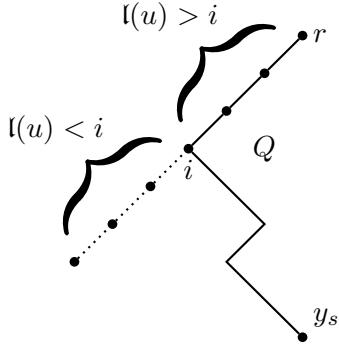
\begin{figure}
  $$
  \begin{tikzpicture}[scale = .5, thick, baseline=0pt]
  \node[inner sep=1pt, draw, circle, fill=black](1) at (0,0){};
  \node[inner sep=1pt, draw, circle, fill=black](2) at (1,1){};
  \node[inner sep=1pt, draw, circle, fill=black](3) at (2,2){};
  \node[inner sep=1pt, draw, circle, fill=black](4) at (3,3){};
  \node[below] at (3,3){$i$};
  \node[inner sep=1pt, draw, circle, fill=black](5) at (4,4){};
  \node[inner sep=1pt, draw, circle, fill=black](6) at (5,5){};
  \node[inner sep=1pt, draw, circle, fill=black](7) at (6,6){};
  \node[right] at (6,6){$r$};
  \draw[dotted] (1) --(2) -- (3) --(4);
  \draw (4) -- (5) -- (6) -- (7);
  \node[scale=2.5,rotate=45] at (0.75,2.25){$\overbrace{}$};
  \node[scale=2.5,rotate=45] at (3.75,5.25){$\overbrace{}$};
  \node at (-0.5,3.5){$\vlabel(u)<i$};
  \node at (2.5,6.5) {$\vlabel(u)>i$};
  \node[inner sep=1pt, draw, circle, fill=black](8) at(6,-2){};
  \node[above right] at (6,-2){$y_s$};
  \draw (4) -- (4,2) -- (5,1) -- (4,0) -- (5,-1) -- (8);
  \node at (5,3){$Q$};
  \end{tikzpicture}
  $$
  \caption{Referring to \textit{Claim 2} of Theorem~\ref{pairs_to_tree},
  the only path $Q$ from $r$ to $y_s$, where $r$ is the root of $T_s$.}\label{figure_claimII}
  \end{figure}

In the following sections, we will use Theorem~\ref{pairs_to_tree}
as a tool to define maps from Fishburn matrices and $(\twoplustwo)$-free
posets to Fishburn trees. Namely, we will identify rows of matrices and
down-sets of posets as the ``elementary blocks'' of a Fishburn cover.
Then, Theorem~\ref{pairs_to_tree} provides a constructive procedure to assemble the resulting blocks in order to obtain a Fishburn tree.
This approach is in fact more general and could be extended to
any other Fishburn structure.

\section{Modified ascent sequences}\label{section_trees_modasc}

Let $x:[n]\to [n]$ be an endofunction. Writing $x_i=x(i)$, as usual, we define
$$
  \asctops(x)=\{(1,x_1)\}\cup\{(i,x_i): 1<i\le n,x_{i-1}< x_i\}
$$
as the set of ascent tops and their indices---including the first
element---and let
$$
  \nub(x) = \{(\min x^{-1}(j), j): 1\leq j\leq \max(x) \}
$$
be the set of first occurrences and their indices.
An \emph{ascent sequence} is an endofunction $x:[n]\to [n]$
such that $x_1=1$ and, for $i\le n-1$,
$$
x_{i+1}\le |\asctops(x_1\cdots x_{i})|+1.
$$
Let $\Ascseq$ be the set of ascent sequences.
Bousquet-M\'elou \etal~\cite{BMCDK} defined an iterative
procedure to map an ascent sequence~$x$ to its modified
version~$\hat{x}$, and the set~$\Modasc$ of \emph{modified ascent
sequences} was originally defined as the image of~$\Ascseq$
under the $x\mapsto\hat{x}$ bijection.
We~\cite{CC} have provided the following characterization of
modified ascent sequences.

\begin{lemma}\label{modasc_char}
  The set $\Modasc$ of modified ascent sequences is characterized by
  $$
    \Modasc=\lbrace x\in\Cay: \asctops(x)=\nub(x)\rbrace.
  $$
\end{lemma}

Alternatively, a recursive definition of $\Modasc$ can be found
in~\cite{CC}, as well as a description of $\Modasc$ by avoidance of
two Cayley-mesh patterns, defined in~\cite{Ce}.

Recall that $\treetoseq(T)$ denotes the in-order sequence of the tree
$T$. We wish to prove that $T\mapsto\treetoseq(T)$ is a bijective
mapping from from Fishburn trees to modified ascent sequences.  We shall
start by showing that the statistics $\unseen(T)$ and $\treetops(T)$ (on
endotrees) are natural analogues of statistics $\nub(x)$ and
$\asctops(x)$ (on endofunctions).

\begin{lemma}\label{nub_unseen}
  Let $T\in\ntrees$ and let $x=\treetoseq(T)$. Then, for each $i\ge 1$,
  $$
    (i,x_i)\in\nub(x)\ \iff\ v_i\in\unseen(T).
  $$
\end{lemma}
\begin{proof}
  Let $i\ge 1$. Then $(i,x_i)\in\nub(x)$ if and only if $x_i$ is the
  leftmost occurrence of the corresponding integer $\vlabel(v_i)$
  in $x$. Equivalently, $\vlabel(v_j)\neq\vlabel(v_i)$ for each $j<i$;
  that is, $v_i\in\unseen(T)$.
\end{proof}

\begin{lemma}\label{asctops_treetops}
  Let $T\in\ntrees$ and let $x=\treetoseq(T)$. Then, for each $i\ge 1$,
  $$
    (i,x_i)\in\asctops(x)\ \iff\ v_i\in\treetops(T(x)).
  $$
\end{lemma}
\begin{proof}
  By definition we have $(1,x_1)\in\asctops(x)$ and
  $v_1\in\treetops(T)$, which takes care of the case $i=1$.
  Assume $i\ge 2$ and suppose, initially, that
  $(i,x_i)\in\asctops(x)$; that is, $x_{i-1}<x_i$. Since $x_{i-1}$ and
  $x_i$ are consecutive entries in the in-order sequence, the node $v_i$
  is visited immediately after $v_{i-1}$ in the in-order traversal of
  $T$. In fact, only the following two cases are admitted:
  \begin{enumerate}
  \item $v_{i-1}$ is the last visited node in the subtree of $T$ with
    root $\lchild(v_i)$. In this case, $\lchild(v_i)\neq\emptyset$ and
    thus $v_i\in\treetops(T)$.
  \item $v_i=\rchild(v_{i-1})$. This is however impossible since $T$ is
    decreasing and $x_i>x_{i-1}$ by our assumptions.
  \end{enumerate}
  For the converse, let $v_i\in\treetops(T)$. Then $v_{i-1}$ is
  contained in the subtree of $T$ with root $\lchild(v_i)$. In
  particular, $x_{i-1}<x_i$ since $T$ is strictly decreasing to the
  left.
\end{proof}

\begin{proposition}
  Let $T$ be an endotree and let $x=\treetoseq(T)$. Then
  $$
    T\in\fishtrees_n\quad\text{if and only if}\quad x\in\Modasc_n.
  $$
\end{proposition}
\begin{proof}
  It follows immediately by Lemma~\ref{modasc_char},
  Lemma~\ref{nub_unseen} and
  Lemma~\ref{asctops_treetops}. Indeed, for each $i\in [n]$,
  \begin{align*}
    (i,x_i)\in\nub(x) &\iff v_i\in\unseen(T)\\
    \shortintertext{and}
    (i,x_i)\in\asctops(x) &\iff v_i\in\treetops(T).
  \end{align*}
  Thus the equality $\asctops(x)=\nub(x)$ is satisfied if and only if
  $\unseen(T)=\treetops(T)$ is satisfied as well.
\end{proof}

\begin{corollary}
  The (restricted) map $\treetoseq:\fishtrees\to\Modasc$ and
  its inverse map $\seqtotree:\Modasc\to\fishtrees$ are size-preserving
  bijections between Fishburn trees and modified ascent sequences.
  In particular, for each $n\ge 1$ we have
  $$
  |\fishtrees_n|=|\Modasc_n|.
  $$
\end{corollary}

A Fishburn tree and its in-order sequence are illustrated in
Figure~\ref{figure_fishburn_tree}.

\section{Fishburn matrices}\label{section_trees_mat}

A \emph{Fishburn matrix} is a lower triangular matrix with non-negative
integer entries such that every row and column contains at least one
nonzero entry. Let $\fishmat$ denote the set of Fishburn matrices. The
size of $A\in\fishmat$ is the sum of its entries, and $\fishmat_n$
denotes the set of Fishburn matrices of size $n$. When displaying a
Fishburn matrix we will for the sake of readability leave the region
above the diagonal empty and zeros on or below the diagonal will be
denoted by a dot. For example, we have
$\fishmat_1= \left\lbrace\begin{bmatrix}1\end{bmatrix}\right\rbrace$,
$$
      \setlength\arraycolsep{3pt}
      \renewcommand\arraystretch{0.75}
      \fishmat_2=
      \left\lbrace
      \begin{bmatrix}
        2
      \end{bmatrix},
      \begin{bmatrix}
        1     &   \\
        \cdot & 1
      \end{bmatrix}
      \right\rbrace
      \,\text{ and }\,
      \setlength\arraycolsep{3pt}
      \renewcommand\arraystretch{0.75}
      \fishmat_3=
      \left\lbrace
      \begin{bmatrix}
        3
      \end{bmatrix},
      \begin{bmatrix}
        2     &   \\
        \cdot & 1
      \end{bmatrix},
      \begin{bmatrix}
        1     &   \\
        \cdot & 2
      \end{bmatrix},
      \begin{bmatrix}
        1 &   \\
        1 & 1
      \end{bmatrix},
      \begin{bmatrix}
        1     &       &   \\
        \cdot & 1     &   \\
        \cdot & \cdot & 1
      \end{bmatrix}
      \right\rbrace.
$$
Our definition of a Fishburn matrix is a slight departure
from the original definition~\cite{DP} in that our matrices are
lower triangular rather than upper triangular.

We wish to define a mapping $\treetomat:\fishtrees\to\fishmat$ by regarding 
the label $\vlabel(u)$ of each node $u$ as a column index and the index
$\blabel(u)$ of the maximal right path containing $u$ as a row index.

Let $T$ be a Fishburn tree with $n$ nodes. Let $k=\max(T)$ and let
$\pairs(T)=B_1\dots B_k$ be the Fishburn cover of $T$. For each
$i,j\in [k]$, let $m_i(j)$ be the multiplicity of $j$ in $B_i$.
Equivalently, $m_i(j)$ is equal to the number of
nodes with label $j$ contained in $W_i$, where $W_i$ is the $i$-th
maximal right path in the $\rpath$-decomposition of $T$. Then we let
$\treetomat(T)$ be the $k\times k$ matrix whose $(i,j)$-th entry is
equal to $m_i(j)$.

A simple high-level description of $\treetomat$ is the following:
\begin{itemize}
  \item Compute the $\rpath$-decomposition of $T$.
  \item Use the $i$-th path $W_i$ to ``fill'' the $i$-th row
  of $\treetomat(T)$.
\end{itemize}

\begin{proposition}
  Let $T$ be a Fishburn tree of size $n$ and with $\max(T)=k$.
  Then $\treetomat(T)$ is a $k\times k$ Fishburn matrix of size $n$.
\end{proposition}
\begin{proof}
  Since $B_i$ is nonempty, the $i$-th row of $\treetomat(T)$ contains at
  least one nonzero entry. Also, $\bigcup_{i\in [k]}B_i=[k]$ and hence each
  column contains at least one nonzero entry. Furthermore,
  $\treetomat(T)$ is lower triangular, since $j\le i$ for each
  $j\in B_i$ and $i\in [k]$. Finally, the number of nodes in $T$ is
  equal to the sum of entries of $\treetomat(T)$.
\end{proof}

For instance, recall the Fishburn cover associated with the
Fishburn tree in Figure~\ref{figure_fishburn_tree}:
$$
  \pairs(T)=\lbrace 1,1 \rbrace\lbrace 1\rbrace\lbrace 1\rbrace
  \lbrace 2,2\rbrace\lbrace 5,5,3\rbrace\lbrace 2\rbrace
  \lbrace 5,5,4,3\rbrace\lbrace 8,8,7,3\rbrace\lbrace 9,6,1\rbrace.
$$
Its corresponding Fishburn matrix $A=\treetomat(T)$ is
$$
  {
      \setlength\arraycolsep{4pt}
      \renewcommand\arraystretch{0.9}
      A=
      \begin{bmatrix}
        2     &       &       &       &       &       &       &       &   \\
        1     & \cdot &       &       &       &       &       &       &   \\
        1     & \cdot & \cdot &       &       &       &       &       &   \\
        \cdot & 2     & \cdot & \cdot &       &       &       &       &   \\
        \cdot & \cdot & 1     & \cdot & 2     &       &       &       &   \\
        \cdot & 1     & \cdot & \cdot & \cdot & \cdot &       &       &   \\
        \cdot & \cdot & 1     & 1     & 2     & \cdot & \cdot &       &   \\
        \cdot & \cdot & 1     & \cdot & \cdot & \cdot & 1     & 2     &   \\
        1     & \cdot & \cdot & \cdot & \cdot & 1     & \cdot & \cdot & 1 \\
      \end{bmatrix}
    }
$$
where, for instance, the penultimate row of $A$ corresponds to the
penultimate block of $\pairs(T)$:
$$
  \begin{bmatrix}
    0 & 0 & 1 & 0 & 0 & 0 & 1 & 2 & 0
  \end{bmatrix}
  \quad\longleftrightarrow\quad \lbrace 8,8,7,3\rbrace.
$$
Defining the inverse map $\mattotree$ is now
straightforward. Let $A=(a_{i,j})$ be a $k\times k$ Fishburn matrix. Let
$\pairs(A)=B_1\dots B_k$, where $B_i$ is a multiset containing
$a_{i,j}$ copies of the integer $j$, for $i,j\in [k]$. Then $\pairs(A)$
is a Fishburn cover. Indeed $\bigcup_{i\in [k]}B_i=[k]$, since $A$
does not contain null columns; each multiset $B_i$ is nonempty, since
$A$ does not contain null rows; and $j\le i$ for each $j\in B_i$, since
$A$ is lower-triangular. Now, due to Theorem~\ref{pairs_to_tree}, there
is a unique Fishburn tree $T$ such that $\pairs(T)=\pairs(A)$, and we
let $\mattotree(A)=T$. Finally, it is clear that $\treetomat(T)=A$.

We have thus proved the following result.

\begin{corollary}
  The map $\treetomat:\fishtrees\to\fishmat$ and its inverse map
  $\mattotree:\fishmat\to\fishtrees$ are
  size-preserving bijections between Fishburn trees and Fishburn
  matrices. In particular, for each $n\ge 1$ we have
  $$
    |\fishtrees_n|=|\fishmat_n|.
  $$
\end{corollary}

Note that $\mattotree(A)$ can be constructed using
Theorem~\ref{pairs_to_tree}: each path $W_i$ is obtained by reading the
entries in the $i$-th row of $A$, and the paths $W_i$ are then assembled
as illustrated in Theorem~\ref{pairs_to_tree}.

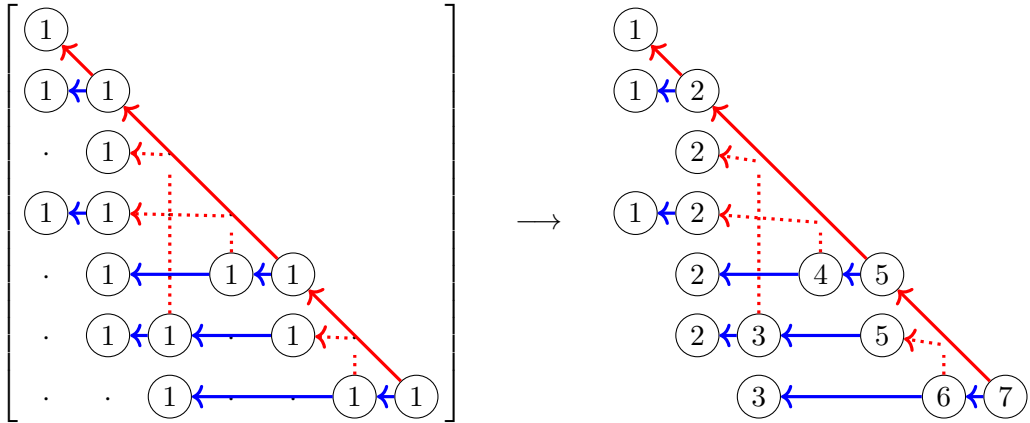
\begin{figure}
  \begin{center}
    \begin{tikzpicture}[baseline=0pt]
      \matrix[
      matrix of math nodes,
      anchor = north west,
      row sep=0.6em,
      column sep=0.6em,
      left delimiter={[},
          right delimiter={]},
          inner xsep=1pt
        ]{
          \node[draw, circle](11){1}; & & & & & & \\
          \node[draw, circle](21){1}; & \node[draw, circle](22){1}; & & & & & \\
          \cdot & \node[draw, circle](32){1};& \node[](33){\cdot}; & & & & \\
          \node[draw, circle](41){1}; & \node[draw, circle](42){1}; & \cdot & \node[](44){\cdot}; & & & \\
          \cdot & \node[draw, circle](52){1}; &\cdot & \node[draw, circle](54){1}; & \node[draw, circle](55){1}; & & \\
          \cdot & \node[draw, circle](62){1}; & \node[draw, circle](63){1}; & \cdot & \node[draw, circle](65){1}; & \node[](66){\cdot}; & \\
          \cdot & \cdot & \node[draw, circle](73){1}; & \cdot & \cdot & \node[draw, circle](76){1}; & \node[draw, circle](77){1}; \\
        };
      \draw[red, ->, very thick] (77) -- (55);
      \draw[red, ->, very thick] (55) --(22);
      \draw[red, ->, very thick] (22) --(11);
      \draw[blue, ->, very thick] (77) -- (76);
      \draw[blue, ->, very thick] (76) --(73);
      \draw[blue, ->, very thick] (65) -- (63);
      \draw[blue, ->, very thick] (63) -- (62);
      \draw[blue, ->, very thick] (55) -- (54);
      \draw[blue, ->, very thick] (54) -- (52);
      \draw[blue, ->, very thick] (42) -- (41);
      \draw[blue, ->, very thick] (22) -- (21);
      \draw[red, very thick, dotted] (76) -- (66);
      \draw[->, red, very thick, dotted] (66) -- (65);
      \draw[red, very thick, dotted] (54) -- (44);
      \draw[->, red, very thick, dotted] (44) -- (42);
      \draw[red, very thick, dotted] (63) -- (33)+(1.5,1.5);
      \draw[->, red, very thick, dotted] (33) -- (32);
      \node at (7,-3){$\longrightarrow\quad$};
    \end{tikzpicture}
    \begin{tikzpicture}[baseline=0pt]
      \matrix[
        matrix of math nodes,
        anchor = north west,
        row sep=0.6em,
        column sep=0.6em,
        inner xsep=1pt
      ]{
        \node[draw, circle](11){1}; &                             &                             &                             &                             &                             &                             \\
        \node[draw, circle](21){1}; & \node[draw, circle](22){2}; &                             &                             &                             &                             &                             \\
                                    & \node[draw, circle](32){2}; & \node[](33){};              &                             &                             &                             &                             \\
        \node[draw, circle](41){1}; & \node[draw, circle](42){2}; &                             & \node[](44){};              &                             &                             &                             \\
                                    & \node[draw, circle](52){2}; &                             & \node[draw, circle](54){4}; & \node[draw, circle](55){5}; &                             &                             \\
                                    & \node[draw, circle](62){2}; & \node[draw, circle](63){3}; &                             & \node[draw, circle](65){5}; & \node[](66){};              &                             \\
                                    &                             & \node[draw, circle](73){3}; &                             &                             & \node[draw, circle](76){6}; & \node[draw, circle](77){7}; \\
      };
      \draw[red, ->, very thick] (77) -- (55);
      \draw[red, ->, very thick] (55) --(22);
      \draw[red, ->, very thick] (22) --(11);
      \draw[blue, ->, very thick] (77) -- (76);
      \draw[blue, ->, very thick] (76) --(73);
      \draw[blue, ->, very thick] (65) -- (63);
      \draw[blue, ->, very thick] (63) -- (62);
      \draw[blue, ->, very thick] (55) -- (54);
      \draw[blue, ->, very thick] (54) -- (52);
      \draw[blue, ->, very thick] (42) -- (41);
      \draw[blue, ->, very thick] (22) -- (21);
      \draw[red, very thick, dotted] (76) -- (66);
      \draw[->, red, very thick, dotted] (66) -- (65);
      \draw[red, very thick, dotted] (54) -- (44);
      \draw[->, red, very thick, dotted] (44) -- (42);
      \draw[red, very thick, dotted] (63) -- (33)+(1.5,1.5);
      \draw[->, red, very thick, dotted] (33) -- (32);
    \end{tikzpicture}
  \end{center}
  \caption{The Fishburn tree illustrated in
    Figure~\ref{figure_comb_tree} drawn on the corresponding binary
    Fishburn matrix. The rightmost entry in the bottom row is the root of the
    tree. Red arrows point to left children and blue arrows point to
    right children. Dotted arrows indicate the ``bouncing'' construction
    that determines the father of non-diagonal rightmost
    entries.}\label{figure_bouncing_entries}
\end{figure}

\begin{remark}
  The Fishburn tree $\mattotree(A)$ can be drawn directly on the
  Fishburn matrix $A$. This construction is most significant on
  binary matrices, where each nonzero entry of $A$ is identified with exactly
  one node of $\treetomat(A)$. Instead of giving the full details, we refer
  the reader to the example in
  Figure~\ref{figure_bouncing_entries}. One interesting aspect is that
  if the rightmost entry of a row is not on the diagonal of $A$, then
  its father can be determined by ``bouncing'' off of the diagonal; indeed
  if $a_{i,i}=0$, then $W_i$ is non-diagonal and the topmost node of
  $W_i$ (i.e.\ the rightmost entry of the $i$-th row of $A$) is the left
  child of a node with label $i$ (i.e.\ in column $i$). In general, more
  than one entry could be hit by bouncing off of the diagonal. To determine
  the correct one is rather tricky, and involves defining a notion of
  in-order traversal of matrices which we have decided to omit.
\end{remark}

\section{$(\twoplustwo)$-free posets}\label{section_trees_pos}

In this paper we consider two posets to be equal up to isomorphism, that is,
if there is an order preserving bijections between them. The isomorphism
class is called an \emph{unlabeled poset}.
An unlabeled poset is \emph{$(\twoplustwo)$-free} if it does not contain an
induced subposet order isomorphic to $\twoplustwo$, the union of two disjoint
$2$-element chains. The \emph{size} of a poset is the number of its
elements and we let $\fishpos_n$ denote the set of unlabeled $(\twoplustwo)$-free
posets of size $n$. Given $Q\in\fishpos=\cup_{n \geq 0}\fishpos_n$ and
$u\in Q$, let
$$
  D(u)=\left\lbrace v: v<u\right\rbrace
$$
be the strict down-set of $u$. Fishburn~\cite{F} showed that a poset is
$(\twoplustwo)$-free if and only if it is order isomorphic to an interval
order. Alternatively (see~\cite{BMCDK} for a proof), a poset is
$(\twoplustwo)$-free if and only if its strict down-sets can be linearly
ordered by inclusion. That is, the strict down-sets of $Q$ form a chain
$$
  \emptyset=D_1\subset D_2\subset \dots\subset D_k.
$$
For convenience, we let $D_{k+1}=Q$.
If $D(u)=D_i$, we say that the element $u$ is at
\emph{level} $i$ and we write $\level(u)=i$. Finally, we let
$$
  L_i=\left\lbrace u: \level(u)=i\right\rbrace
$$
denote the $i$-th level of $Q$. It is clear that a $(\twoplustwo)$-free
poset is completely determined by its levels and strict down-sets.
Indeed, any poset $Q$ is determined by the list of its strict down-sets
$\lbrace D(u): u\in Q\rbrace$, and $D(u)=D_i$ if $\level(u)=i$.
An element $u$ of $Q$ is \emph{maximal} if no other element of $Q$
is greater than $u$. It is \emph{minimal} if no other element is
smaller than $u$. We let $\max(Q)$ and $\min(Q)$
denote the set of maximal and minimal elements, respectively.
It is easy to see that if $Q$ is $(\twoplustwo)$-free, then
$\min(Q)=L_1$ and $\max(Q)=D_{k+1}\setminus D_k$.

We wish to define a bijection $\treetopos: \fishtrees\to\fishpos$.
Let $T$ be a Fishburn tree and let $k=\max(T)$.
Recall from Definition~\ref{def_of_b} that, given $u\in\vertset(T)$, the
index of the maximal right path that contains $u$ in the
$\rpath$-decomposition of $T$ is denoted by $\blabel(u)$.
Recall also that $\vlabel(u)\le\blabel(u)$ for each $u\in\vertset(T)$,
a fact that will be used repeatedly in this section.
We wish to define a $(\twoplustwo)$-free poset $Q=\treetopos(T)$ by
associating each node of $T$ with an element of $Q$. That is,
we let $V(T)$ be the set of elements of $Q$. Then we define a partial
order on $Q$ by letting, for any $u$ and $v$ in $Q$,
$$
u<v\iff\blabel(u)<\vlabel(v).
$$
Let us prove that this relation is a strict partial order.
Irreflexivity is an immediate consequence of
the inequality $\vlabel(u)\le\blabel(u)$. To prove antisymmetry, suppose
that $u<v$; i.e.\ $\blabel(u)<\vlabel(v)$. For a contradiction,
suppose also that $v<u$; i.e.\ $\blabel(v)<\vlabel(u)$. Then
$$
\blabel(u)<\vlabel(v)\le\blabel(v)<\vlabel(u),
$$
from which we get $\blabel(u)<\vlabel(u)$, which is impossible.
Finally, to prove transitivity, suppose that $u<v$ and $v<w$. Then
$$
\blabel(u)<\vlabel(v)\le\blabel(v)<\vlabel(w),
$$
from which $\blabel(u)<\vlabel(w)$, and thus $u<w$, follows.
To prove that $Q$ is $(\twoplustwo)$-free, we show that its strict down-sets
are linearly ordered by inclusion. Let us first determine its strict
down-sets. Let $u\in Q$ and suppose that $\vlabel(u)=i$. The strict
down-set of $u$ is
$$
D(u)=\lbrace v\in Q: \blabel(v)<i\rbrace.
$$
In other words, all the elements with vertex label $i$ have the same
strict down-set, namely $\lbrace v\in Q: \blabel(v)<i\rbrace$. For $i\in [k]$,
let $D_i=\lbrace v\in Q: \blabel(v)<i\rbrace$. Note that there is at least
one element in $Q$ whose down-set is $D_i$ since $\vlabel(Q)=[k]$ by the
definition of Fishburn tree.
Now, it is clear by definition that $D_i\subseteq D_{i+1}$.
Furthermore, the inclusion is strict since $\blabel(Q)=[k]$;
thus, there is at least one element in
$D_{i+1}\setminus D_i=\lbrace u\in Q:\blabel(u)=i\rbrace$.
Therefore, the down-sets of $Q$ are precisely the sets $D_i$, $i\in [k]$,
which are strictly ordered by inclusion. We have now proved that $Q$
is $(\twoplustwo)$-free. Note that the levels of $Q$ are
$$
L_i=
\lbrace u\in Q: D(u)=D_i\rbrace=
\lbrace u\in Q: \vlabel(u)=i\rbrace.
$$
In fact, an alternative way to define $Q$ is to let its levels
and strict down-sets be
$$
  L_i=\lbrace u\in\vertset(T): \vlabel(u)=i\rbrace
  \quad\text{and}\quad
  D_i=\lbrace u\in\vertset(T): \blabel(u) < i\rbrace.
$$

The inverse map $\postotree$ of $\treetopos$ is defined as follows. Given a poset
$Q\in\fishpos$, we define a canonical labeling of $Q$ by setting, for
each $u\in Q$,
$$
  \vlabel(u)=\level(u)
  \quad\text{and}\quad
  \blabel(u)=\min\lbrace i:u\in D_i\rbrace-1.
$$
Note that
$\lbrace\vlabel(u):u\in Q\rbrace=\lbrace\blabel(u):u\in Q\rbrace=[k]$,
where $k$ is the number of levels of $Q$. Moreover, we have
$\vlabel(u)\le\blabel(u)$ for each $u\in Q$. In fact, to the poset $Q$
we have associated the Fishburn cover $\pairs(Q)=B_1\dots B_k$,
where $B_i$ contains a copy of the integer $j$ for each $u\in Q$ with
labels $\blabel(u)=i$ and $\vlabel(u)=j$. We can thus use
Theorem~\ref{pairs_to_tree} to construct a Fishburn tree
$T=\postotree(Q)$ in which each node $u$ has labels $\vlabel(u)$ and
$\blabel(u)$. Finally, it is easy to see that $\treetopos(T)=Q$ and it
follows that $\postotree$ is the inverse map of $\treetopos$. Indeed,
$\vlabel(u)=\level(u)$ and $\blabel(u)=i+1-1=i$, where
$u\in D_{i+1}\setminus D_i$. In the end we obtain the following result.

\begin{corollary}
  The map $\treetopos:\fishtrees\to\fishpos$ and its inverse map
  $\postotree:\fishpos\to\fishtrees$ are size-preserving bijections
  between Fishburn trees and $(\twoplustwo)$-free posets.
  In particular, for each $n\ge 1$ we have
  $$
  |\fishtrees_n|=|\fishpos_n|.
  $$
\end{corollary}

A Fishburn tree and the canonical labeling of
the corresponding poset are illustrated in Figure~\ref{figure_poset}.

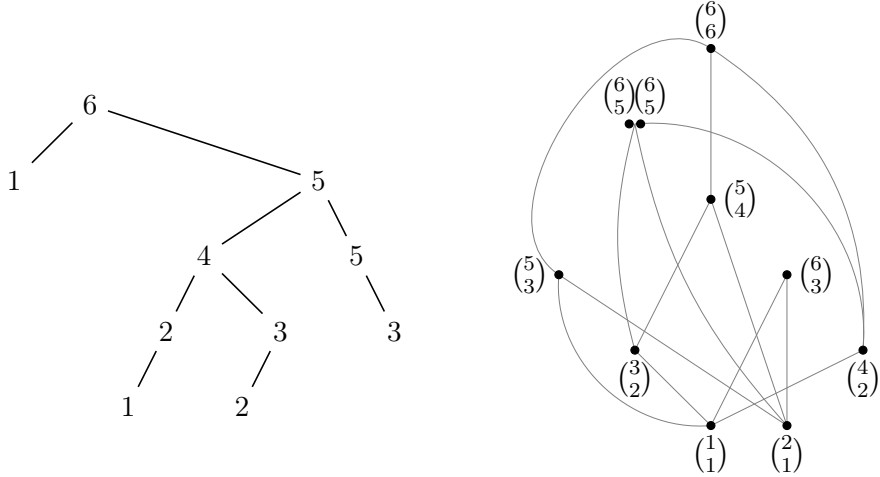
\begin{figure}
$$
\begin{tikzpicture}[scale = .5, semithick, baseline=0pt, level distance=2cm]
\tikzstyle{level 1}=[sibling distance=4cm]
\tikzstyle{level 2}=[sibling distance=4cm]
\tikzstyle{level 3}=[sibling distance=2cm]
\tikzstyle{level 4}=[sibling distance=2cm]
\node at(0,10){6}
	child {node{1}}
	child {node at(4,0){5}
		child{node at(-1,0){4}
				child{node{2}
					child{node at(-1,0){1}}
				}			
				child{nodeat(1,0){3}
					child{nodeat(-1,0){2}}				
				}	
	}
		child{nodeat(-1,0){5}
			child{nodeat(1,0){3}}		
		}
	}
;
\end{tikzpicture}
\qquad
\begin{tikzpicture}[scale=0.5]
\draw[ultra thin, gray] (2,2)edge[bend right=00](4,0);
\draw[ultra thin, gray] (8,2)edge[bend right=00](4,0);
\draw[ultra thin, gray] (0,4)edge[bend right=50](4,0);
\draw[ultra thin, gray] (0,4)edge[bend left=00](6,0);
\draw[ultra thin, gray] (6,4)edge[bend right=00](4,0);
\draw[ultra thin, gray] (6,4)edge[bend right=00](6,0);
\draw[ultra thin, gray] (4,6)edge[bend right=00](2,2);
\draw[ultra thin, gray] (4,6)edge[bend right=00](6,0);
\draw[ultra thin, gray] (2,8)edge[bend right=15](2,2);
\draw[ultra thin, gray] (2,8)edge[bend right=15](6,0);
\draw[ultra thin, gray] (2,8)edge[bend left=50](8,2);
\draw[ultra thin, gray] (4,10)edge[bend right=90](0,4);
\draw[ultra thin, gray] (4,10)edge[bend right=00](4,6);
\draw[ultra thin, gray] (4,10)edge[bend left=30](8,2);
\filldraw [black] (4,0) circle (3pt) node[below]{$\binom{1}{1}$};
\filldraw [black] (6,0) circle (3pt) node[below]{$\binom{2}{1}$};
\filldraw [black] (2,2) circle (3pt) node[below]{$\binom{3}{2}$};
\filldraw [black] (8,2) circle (3pt) node[below]{$\binom{4}{2}$};
\filldraw [black] (0,4) circle (3pt) node[left]{$\binom{5}{3}$};
\filldraw [black] (6,4) circle (3pt) node[right]{$\binom{6}{3}$};
\filldraw [black] (4,6) circle (3pt) node[right]{$\binom{5}{4}$};
\filldraw [black] (1.85,8) circle (3pt) node[above,xshift=-0.125cm]{$\binom{6}{5}$};
\filldraw [black] (2.15,8) circle (3pt) node[above,xshift=0.125cm]{$\binom{6}{5}$};
\filldraw [black] (4,10) circle (3pt) node[above]{$\binom{6}{6}$};
\end{tikzpicture}
$$
\caption{A Fishburn tree $T$ and the poset $\treetopos(T)$.
The poset is equipped with its canonical labeling; that is, each node
$u$ is equipped with the pair of labels $\binom{\blabel(u)}{\vlabel(u)}$.}\label{figure_poset}
\end{figure}

\section{Flip and sum operations on $\Modasc$}\label{section_flipsum}

Duality acts as an
involution on $(\twoplustwo)$-free posets. Dukes and Parviainen~\cite{DP}
showed that this operation is equivalent to computing the reflection of a
Fishburn matrix in its antidiagonal. On the other hand, it is difficult
to infer how duality acts on the corresponding ascent
sequences. Similarly, the sum of two Fishburn matrices
is a Fishburn matrix, but to describe the corresponding sum
operation on ascent sequences is a challenging problem. In this section
we use Fishburn covers to provide a more direct construction for
both problems in terms of modified ascent sequences.

Let $A=(a_{i,j})$ be a $k\times k$ matrix. We denote by $\flip(A)$ the
reflection of $A$ in its antidiagonal; that is, the $(i,j)$-th entry of
$\flip(A)$ is equal to
$$
  \flip(A)(i,j)=a_{k+1-j,k+1-i}.
$$

Let $A=(a_{i,j})$ and $A'=(a'_{i,j})$ be two matrices of dimension
$p\times p$ and $q\times q$, respectively, with $p\le q$. We denote by
$A+A'$ the $q\times q$ matrix obtained by summing $A$ and $A'$
entry by entry; that is, the $(i,j)$-th entry of $A+A'$ is equal
to
$$
  (A+A')(i,j)=
  \begin{cases}
    a_{i,j}+a'_{i,j},\quad & \text{if }i\le p\text{ and } j\le p; \\
    a'_{i,j},\quad         & \text{if }i>p\text{ or }j>p.
  \end{cases}
$$

It is easy to see that if $A$ is a Fishburn matrix, then
$\flip(A)$ is a Fishburn matrix as well. Similarly, the sum $A+A'$
of two Fishburn matrices $A$ and $A'$ is a Fishburn matrix. The $\flip$
and $\sumop$ problems are formulated in terms of modified ascent
sequences as follows:

\begin{itemize}
\item Let $x$ be a modified ascent sequence and let
  $A=(\treetomat\circ\seqtotree)(x)$ be the corresponding Fishburn
  matrix. What is the modified ascent sequence $\flip(x)$ that
  corresponds to $\flip(A)$?
\item Let $x$ and $x'$ be modified ascent sequences and let
  $A=(\treetomat\circ\seqtotree)(x)$ and
  $A'=(\treetomat\circ\seqtotree)(x')$ be the corresponding Fishburn
  matrices. What is the modified ascent sequence $x+x'$ that
  corresponds to $A+A'$?
\end{itemize}

An answer to the previous two questions could be obtained by composing
the bijection $\treetoseq$, defined in Section~\ref{section_prelim},
with the bijection $\treetomat$, defined in Section~\ref{section_trees_mat}.
For instance, we could first determine the Fishburn matrix
$A=\treetomat\bigl(\seqtotree(x)\bigr)$ associated with the modified ascent
sequence~$x$, then compute $\flip(A)$, and finally obtain $\flip(x)$ as
$\treetoseq\bigl(\mattotree(\flip(A))\bigr)$.
However, we have defined $\treetoseq$ in terms of Fishburn trees, while
$\treetomat$ was defined in terms of Fishburn covers.
To make the whole construction more straightforward, we wish to reinterpret
$\treetoseq$ and its inverse $\seqtotree$ in terms of Fishburn covers
(see also Figure~\ref{figure_diagram_flip_sum}).

\begin{figure}
  $$
    \begin{tikzcd}
      x \arrow{r}{\seqtotree} & \pairs(x) \arrow{d}{\flip} \\
      \flip(x) & \flip\bigl(\pairs(x)\bigr) \arrow{l}{\treetoseq}
    \end{tikzcd}
    \qquad
    \begin{tikzcd}[row sep = tiny]
      x \arrow{r}{\seqtotree} & \pairs(x) \arrow{dr} & & \\
      & & \pairs(x)\oplus\pairs(x') \arrow{r}{\treetoseq} & x+x' \\
      x'\arrow{r}{\seqtotree} & \pairs(x')\arrow{ur} & &
    \end{tikzcd}
  $$
  \caption{Diagrams to compute the flip and sum operations on modified ascent sequences.}\label{figure_diagram_flip_sum}
\end{figure}

Let $x\in\Modasc$ and $A\in\fishmat$. With slight abuse of notation,
we denote by $\pairs(x)$ the Fishburn cover of $x$; that is, we let
$\pairs(x)=\pairs\bigl(\seqtotree(x)\bigr)$. Similarly, we let
$\pairs(A)=\pairs\bigl(\mattotree(A)\bigr)$ be the Fishburn cover of $A$.
Our first goal is to describe the composition
$$
  x\quad\loongmapsto{$\seqtotree$}\quad\pairs(x)\quad\loongmapsto{$\treetomat$}\quad A
$$
and its inverse
$$
  A\quad\loongmapsto{$\mattotree$}\quad\pairs(A)\quad\loongmapsto{$\treetoseq$}\quad x,
$$
thus bypassing the construction of the intermediate Fishburn tree. We
spell out the main ideas below, leaving some details to the reader.

We start by redefining~$\treetoseq:\fishtrees\to\Modasc$ in terms of
Fishburn covers. Let $P=B_1\dots B_k$ be a Fishburn cover. For each $i$,
let $\overrightarrow{B_i}$ be the sequence obtained by arranging $B_i$ in weakly
decreasing order. Following Theorem~\ref{pairs_to_tree}, let
$$
  D=\lbrace i\in [k]: i\in B_i\rbrace\quad\text{and}\quad\bar{D}=\lbrace i\in [k]: i\notin B_i\rbrace.
$$
Write
$$
  D=\lbrace i_1,i_2,\dots, i_p\rbrace\quad\text{and}\quad\bar{D}=\lbrace j_1,j_2,\dots,j_q\rbrace,
$$
with $p+q=k$, $i_1<i_2<\dots<i_p$ and $j_1>j_2>\dots>j_m$. The modified
ascent sequence~$x$ associated with~$P$ is defined as follows:
\begin{enumerate}
\item Define $x^{(0)}=\overrightarrow{B_{i_{1}}}\overrightarrow{B_{i_{2}}}\dots
  \overrightarrow{B_{i_{p}}}$ as the sequence obtained by juxtaposing
  the ``diagonal blocks''.
\item For $s=1,2,\dots,q$, let $x^{(s)}$ be obtained from $x^{(s-1)}$ by
  inserting $\overrightarrow{B_s}$ immediately before the leftmost occurrence
  of the integer $j_s$.
\end{enumerate}
Finally, the desired modified sequence is $x=x^{(q)}$. Referring once
again to Theorem~\ref{pairs_to_tree}, the initial sequence $x^{(0)}$ is
the in-order sequence of the ``comb-shaped'' tree $T_0$. The second item
produces a succession of sequences
$x^{(0)}\subset x^{(1)}\subset\cdots\subset x^{(q)}$, where $x^{(s)}$ is
the in-order sequence of $T_s$, for $s=1,2,\dots,q$. In particular, the
insertion of $\overrightarrow{B_s}$ is analogous to the operation of
attaching $T_s$: each insertion creates a new ascent top; ascent tops
have distinct labels; and ascent tops are preserved when new blocks
are appended.
To illustrate this construction, let
$$
  P=
  \lbrace 1\rbrace
  \lbrace 2,1\rbrace
  \lbrace 2\rbrace
  \lbrace 2,1\rbrace
  \lbrace 5,4,2\rbrace
  \lbrace 5,3,2\rbrace
  \lbrace 7,6,3\rbrace.
$$
be the Fishburn cover of Figure~\ref{figure_comb_tree}. The diagonal
blocks are
$$
  \overrightarrow{B_1}=1,\quad \overrightarrow{B_2}=21,\quad
  \overrightarrow{B_5}=542,\quad \overrightarrow{B_7}=763.
$$
The non-diagonal blocks are
$$
  \overrightarrow{B_3}=2,\quad \overrightarrow{B_4}=21,\quad
  \overrightarrow{B_6}=532.
$$
By juxtaposing the diagonal blocks, we obtain
$$
  x^{(0)}=1\ 21\ 542\ 763.
$$
Then we insert the non-diagonal blocks, each one immediately before
the leftmost occurrence of its index, starting from the one with
biggest index:
\begin{align*}
  x^{(0)} & = 121542763\\
  \overrightarrow{B_6}\;\longrightarrow\quad x^{(1)} & = 1215427\ \underline{532}\ 63\\
  \overrightarrow{B_4}\;\longrightarrow\quad x^{(2)} & = 1215\ \underline{21}\ 427\ 532\ 63\\
  \overrightarrow{B_3}\;\longrightarrow\quad x^{(3)} & = 1215\ 21\ 427\ 5\ \underline{2}\ 32\ 63\\
\end{align*}
In the end, we get the modified ascent sequence
$$
  x=x^{(3)}=121521427523263.
$$
As expected, $x$ is the in-order sequence of
the Fishburn tree of Figure~\ref{figure_comb_tree}.

Conversely, we wish to define the Fishburn cover $\pairs(x)$
directly on $x$. Equivalently, for each entry $x_i$ we determine the
index $\blabel(x_i)$ of the maximal right path that contains the corresponding
node $v_i$ in $\seqtotree(x)$. To do so, we recursively apply the
max-decomposition to $x$, as in the definition of $\seqtotree$. The
label of the leftmost occurrence $x_m$ of $\max(x)$ is
$\blabel(x_m)=x_m$. Let $y=\pref(y)y_m\suff(y)$ be the current
sequence in the max-decomposition of $x$.
Then the leftmost occurrence $y_j$ of $\max(\pref(y))$ in $\pref(y)$ gets label
\begin{align*}
  \blabel(y_j) &=
    \begin{cases}
      y_j & \text{if $y_m$ is a left-to-right maximum of $x$,} \\
      y_m & \text{otherwise.}
    \end{cases}\\
\intertext{While the leftmost occurrence $y_j$ of $\max(\suff(y))$ in $\suff(y)$ gets label}
  \blabel(y_j) &= \blabel(y_m).
\end{align*}

It is not hard to see that these rules are analogous to the rules
given in Equation~\ref{eqn:recursive_def_of_b} and illustrated in
Figure~\ref{figure_rules_blabel}. Below we apply this procedure
to the modified ascent sequence $x=121521427523263$ obtained before.
At each step, the current leftmost maxima are highlighted; arrows
starting from a current leftmost maximum carry the $\blabel$-label
of the target node; and $\blabel$-labels are recorded as subscripts.
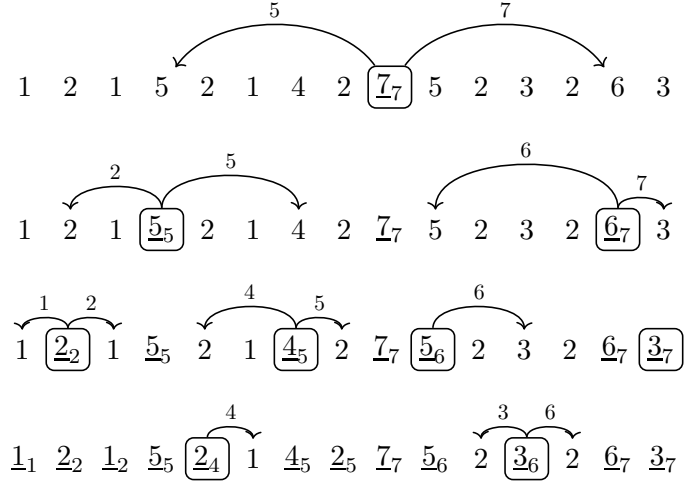
\begin{center}
  \tikzstyle{gstyle}=[draw,semithick,inner sep=3pt,rounded corners=3pt]
    \begin{tabular}{c}
    \begin{tikzpicture}[xscale=0.6, yscale=0.5, semithick, baseline=0pt]
      \node(1) at (1,0){$1$};
      \node(2) at (2,0){$2$};
      \node(3) at (3,0){$1$};
      \node(4) at (4,0){$5$};
      \node(5) at (5,0){$2$};
      \node(6) at (6,0){$1$};
      \node(7) at (7,0){$4$};
      \node(8) at (8,0){$2$};
      \node(9) at (9,0) [gstyle] {$\underline{7}_7$};
      \node(10) at (10,0){$5$};
      \node(11) at (11,0){$2$};
      \node(12) at (12,0){$3$};
      \node(13) at (13,0){$2$};
      \node(14) at (14,0){$6$};
      \node(15) at (15,0){$3$};
      \draw (9) edge[->,bend right=60] node[above, scale=.75]{$5$} (4);
      \draw (9) edge[->,bend left=60]  node[above, scale=.75]{$7$} (14);
    \end{tikzpicture}
    \\[3ex]
    \begin{tikzpicture}[xscale=0.6, yscale=0.5, semithick, baseline=0pt]
      \node(1) at (1,0){$1$};
      \node(2) at (2,0){$2$};
      \node(3) at (3,0){$1$};
      \node(4) at (4,0) [gstyle] {$\underline{5}_5$};
      \node(5) at (5,0){$2$};
      \node(6) at (6,0){$1$};
      \node(7) at (7,0){$4$};
      \node(8) at (8,0){$2$};
      \node(9) at (9,0){$\underline{7}_7$};
      \node(10) at (10,0){$5$};
      \node(11) at (11,0){$2$};
      \node(12) at (12,0){$3$};
      \node(13) at (13,0){$2$};
      \node(14) at (14,0) [gstyle] {$\underline{6}_7$};
      \node(15) at (15,0){$3$};
      \draw (4)  edge[->,bend right=90] node[above, scale=.75]{$2$} (2);
      \draw (14) edge[->,bend right=90] node[above, scale=.75]{$6$} (10);
      \draw (4)  edge[->,bend left=90]  node[above, scale=.75]{$5$} (7);
      \draw (14) edge[->,bend left=90]  node[above, scale=.75]{$7$} (15);
    \end{tikzpicture}
    \\[3ex]
    \begin{tikzpicture}[xscale=0.6, yscale=0.5, semithick, baseline=0pt]
      \node(1) at (1,0){$1$};
      \node(2) at (2,0) [gstyle] {$\underline{2}_2$};
      \node(3) at (3,0){$1$};
      \node(4) at (4,0){$\underline{5}_5$};
      \node(5) at (5,0){$2$};
      \node(6) at (6,0){$1$};
      \node(7) at (7,0) [gstyle] {$\underline{4}_5$};
      \node(8) at (8,0){$2$};
      \node(9) at (9,0){$\underline{7}_7$};
      \node(10) at (10,0) [gstyle] {$\underline{5}_6$};
      \node(11) at (11,0){$2$};
      \node(12) at (12,0){$3$};
      \node(13) at (13,0){$2$};
      \node(14) at (14,0){$\underline{6}_7$};
      \node(15) at (15,0) [gstyle] {$\underline{3}_7$};
      \draw (2)  edge[->,bend right=90] node[above, scale=.75]{$1$} (1);
      \draw (7)  edge[->,bend right=90] node[above, scale=.75]{$4$} (5);
      \draw (2)  edge[->,bend left=90]  node[above, scale=.75]{$2$} (3);
      \draw (7)  edge[->,bend left=90]  node[above, scale=.75]{$5$} (8);
      \draw (10) edge[->,bend left=90]  node[above, scale=.75]{$6$} (12);
    \end{tikzpicture}
    \\[3ex]
    \begin{tikzpicture}[xscale=0.6, yscale=0.5, semithick, baseline=0pt]
      \node(1) at (1,0){$\underline{1}_1$};
      \node(2) at (2,0){$\underline{2}_2$};
      \node(3) at (3,0){$\underline{1}_2$};
      \node(4) at (4,0){$\underline{5}_5$};
      \node(5) at (5,0) [gstyle] {$\underline{2}_4$};
      \node(6) at (6,0){$1$};
      \node(7) at (7,0){$\underline{4}_5$};
      \node(8) at (8,0){$\underline{2}_5$};
      \node(9) at (9,0){$\underline{7}_7$};
      \node(10) at (10,0){$\underline{5}_6$};
      \node(11) at (11,0){$2$};
      \node(12) at (12,0) [gstyle] {$\underline{3}_6$};
      \node(13) at (13,0){$2$};
      \node(14) at (14,0){$\underline{6}_7$};
      \node(15) at (15,0){$\underline{3}_7$};
      \draw (12) edge[->,bend right=90] node[above, scale=.75]{$3$} (11);
      \draw (5)  edge[->,bend left=90]  node[above, scale=.75]{$4$} (6);
      \draw (12) edge[->,bend left=90]  node[above, scale=.75]{$6$} (13);
    \end{tikzpicture}
    \end{tabular}
\end{center}
In the end, we get
 $$
  x=1_1\ 2_2\ 1_2\ 5_5\ 2_4\ 1_4\ 4_5\ 2_5\ 7_7\ 5_6\ 2_3\ 3_6\ 2_6\ 6_7\ 3_7,
 $$
and, as expected, the corresponding Fishburn cover is
$$
\pairs(x)=\lbrace 1\rbrace\lbrace 2,1\rbrace\lbrace 2\rbrace\lbrace 2,1\rbrace\lbrace 5,4,2\rbrace\lbrace 5,3,2\rbrace
    \lbrace 7,6,3\rbrace.
$$
We are now able to compute the flip
and sum operations on modified ascent sequences. For convenience, we
represent a Fishburn cover $P=B_1\dots B_k$ as a biword containing
a column $\binom{i}{j}$ for each $j\in B_i$, with entries in the top row
sorted in increasing order, and breaking ties by sorting the bottom row
in decreasing order. For instance, the Fishburn cover obtained above is
written as
$$
{
 \setlength\arraycolsep{2pt}
 \setcounter{MaxMatrixCols}{25}
 \pairs(x)=
 \begin{pmatrix}
  \ 1\  & 2 & 2\  & 3\  & 4 & 4\  & 5 & 5 & 5\  & 6 & 6 & 6\  & 7 & 7 & 7 \ \\
  \ 1\  & 2 & 1\  & 2\  & 2 & 1\  & 5 & 4 & 2\  & 5 & 3 & 2\  & 7 & 6 & 3 \ \\
 \end{pmatrix}.
}
$$
A biword whose entries are sorted this way is called a
\emph{Burge word}~\cite{AlUh,CC}. It is well known that Burge words are in
bijection with nonnegative integer matrices whose every row and column has
at least one nonzero entry: each biword is associated to a matrix whose
$(i,j)$-th entry is equal to the number of columns $\binom{i}{j}$
contained in the biword. The map $\treetomat$ is simply the restriction
of this correspondence on Fishburn covers and Fishburn matrices.

\subsection*{The flip operation}\label{section_flip}
Let $A=(a_{i,j})$ be a $k\times k$ Fishburn
matrix. By applying the $\flip$ operation, the $(i,j)$-th entry of $A$
is mapped to the $(k+1-j,k+1-i)$-th entry of $\flip(A)$. In terms of
Fishburn covers, the flip operation acts on the columns of
$\pairs(A)$ by
$$
  \binom{i}{j}\longmapsto\binom{k+1-j}{k+1-i}.
$$
Given a Fishburn cover $P$, let $\flip(P)$ be the biword obtained by
applying the above operation to each column of $P$ (and then
  sorting the resulting columns accordingly).
It is clear from the preceding discussion that
$\flip(x)$ is the modified sequence of
$\flip\bigl(\pairs(x)\bigr)$; that is,
$$
  \flip(x)=\treetoseq\bigl( \flip\bigl(\pairs(x)\bigr) \bigr)
$$

\begin{example}\label{example_flip}
  Let $x=1612423553$ be a modified ascent sequence (note that
  this is the in-order sequence of the Fishburn tree depicted
  in Figure~\ref{figure_poset}). We shall compute
  $\flip(x)$. The Fishburn cover of $x$ is
  $$
    {
        \setlength\arraycolsep{2pt}
        \pairs(x)=
        \begin{pmatrix}
          \ 1\ & 2\ & 3\ & 4\ & 5 & 5\ & 6 & 6 & 6 & 6 \ \\
          \ 1\ & 1\ & 2\ & 2\ & 4 & 3\ & 6 & 5 & 5 & 3 \ \\
        \end{pmatrix}
      }
  $$
  and
  $$
    {
        \setlength\arraycolsep{2pt}
        \flip\bigl(\pairs(x)\bigr)=
        \begin{pmatrix}
          \ 1\ & 2 & 2\ & 3\ & 4 & 4\ & 5 & 5\ & 6 & 6 \ \\
          \ 1\ & 1 & 1\ & 2\ & 2 & 1\ & 4 & 3\ & 6 & 5 \ \\
        \end{pmatrix}.
      }
  $$
  Finally, we apply $\treetoseq$ to obtain
  $$
    \flip(x)=1611214235.
  $$
  To check that $\flip(x)$ is the correct sequence, it is easy to
  compute the corresponding matrices
  $$
    {
        \setlength\arraycolsep{3pt}
        \renewcommand\arraystretch{0.75}
        \bigl(\treetomat\circ\seqtotree\bigr)(x)=
        \begin{bmatrix}
          1     &       &       &       &       &   \\
          1     & \cdot &       &       &       &   \\
          \cdot & 1     & \cdot &       &       &   \\
          \cdot & 1     & \cdot & \cdot &       &   \\
          \cdot & \cdot & 1     & 1     & \cdot &   \\
          \cdot & \cdot & 1     & \cdot & 2     & 1 \\
        \end{bmatrix}
        \quad\text{and}\quad
        \bigl(\treetomat\circ\seqtotree\bigr)\bigl(\flip(x)\bigr)=
        \begin{bmatrix}
          1     &       &       &       &       &   \\
          2     & \cdot &       &       &       &   \\
          \cdot & 1     & \cdot &       &       &   \\
          1     & 1     & \cdot & \cdot &       &   \\
          \cdot & \cdot & 1     & 1     & \cdot &   \\
          \cdot & \cdot & \cdot & \cdot & 1     & 1 \\
        \end{bmatrix}
      }
  $$
  to see that each one is the flip of the other.
\end{example}

\subsection*{The sum operation} Let $A$ and $A'$ be two Fishburn matrices of
dimension $j\times j$ and $k\times k$, respectively, with $j<k$.  Let
$$
  \renewcommand\arraystretch{0.94}
  \pairs(A)=\begin{pmatrix}
    1\cdots 1 & 2\cdots 2 & \dots & j\cdots j \\
    \overrightarrow{B_1} & \overrightarrow{B_2} & \dots & \overrightarrow{B_k}
  \end{pmatrix}
  \text{ and }\,
  \pairs(A')=\begin{pmatrix}
    1\cdots 1 & 2\cdots 2 & \dots & k\cdots k \\
    \overrightarrow{B'_1} & \overrightarrow{B'_2} & \dots & \overrightarrow{B'_k}
  \end{pmatrix}\!.
$$
It is easy to see that the Fishburn cover associated to
$A+A'$ contains the union of columns of $A$ and $A'$. We define a
sum operation on Fishburn covers accordingly. That is, if
$P=B_1\cdots B_j$ and $P'=B'_1\cdots B'_k$ are Fishburn covers, we let
$$
  P+P'=\begin{pmatrix}
    1\cdots 1    & 2\cdots 2    & \dots & j\cdots j    & j+1\cdots j+1 & \dots & k\cdots k \\
    \overrightarrow{B_1\cup B'_1} & \overrightarrow{B_2\cup B'_2} & \dots & \overrightarrow{B_j\cup B'_j} & \overrightarrow{B'_{j+1}}      & \dots & \overrightarrow{B'_{k}}
  \end{pmatrix}.
$$
Therefore, $x+x'$ is the modified sequence of
$\pairs(x)+\pairs(x')$; that is,
$$
  x+x'=\treetoseq\bigl( \pairs(x)+\pairs(x') \bigr)
$$

\begin{example}
  Let $x=1612423553$ and $x'=113312443$ be modified ascent sequences. We
  wish to compute their sum $x+x'$. We have
  $$
    \pairs(x)=
    {
    \setlength\arraycolsep{2pt}
    \begin{pmatrix}
      \ 1\ & 2\ & 3\ & 4\ & 5 & 5\ & 6 & 6 & 6 & 6 \ \\
      \ 1\ & 1\ & 2\ & 2\ & 4 & 3\ & 6 & 5 & 5 & 3 \ \\
    \end{pmatrix}
    \quad\text{and}\quad
    \pairs(x')=
    \begin{pmatrix}
      \ 1 & 1\ & 2\ & 3 & 3 & 3\ & 4 & 4 & 4 \ \\
      \ 1 & 1\ & 1\ & 3 & 3 & 2\ & 4 & 4 & 3 \ \\
    \end{pmatrix}.
    }
  $$
  Thus
  $$
    {
        \setlength\arraycolsep{2pt}
        \pairs(x)+\pairs(x')=
        \begin{pmatrix}
          \ 1 & 1 & 1\ & 2 & 2\ & 3 & 3 & 3 & 3\ & 4 & 4 & 4 & 4\ & 5 & 5\ & 6 & 6 & 6 & 6 \ \\
          \ 1 & 1 & 1\ & 1 & 1\ & 3 & 3 & 2 & 2\ & 4 & 4 & 3 & 2\ & 4 & 3\ & 6 & 5 & 5 & 3 \ \\
        \end{pmatrix}
      }
  $$
  and
  $$
    x+x'=\treetoseq\bigl(\pairs(x)+\pairs(x')\bigr)=
    1113311224432643553.
  $$
  Again, it is easy to check that the equality
  $\bigl(\treetomat\circ\seqtotree\bigr)(x+x')=
  \bigl(\treetomat\circ\seqtotree\bigr)(x)+
  \bigl(\treetomat\circ\seqtotree\bigr)(x')$
  between the corresponding matrices holds. Indeed,
  $$
    {
        \setlength\arraycolsep{3pt}
        \renewcommand\arraystretch{0.75}
        \bigl(\treetomat\circ\seqtotree\bigr)(x)=
        \begin{bmatrix}
          1     &       &       &       &       &   \\
          1     & \cdot &       &       &       &   \\
          \cdot & 1     & \cdot &       &       &   \\
          \cdot & 1     & \cdot & \cdot &       &   \\
          \cdot & \cdot & 1     & 1     & \cdot &   \\
          \cdot & \cdot & 1     & \cdot & 2     & 1 \\
        \end{bmatrix},
        \qquad
        \bigl(\treetomat\circ\seqtotree\bigr)(x')=
        \begin{bmatrix}
          2     &       &   &   \\
          1     & \cdot &   &   \\
          \cdot & 1     & 2 &   \\
          \cdot & \cdot & 1 & 2 \\
        \end{bmatrix}
      }
  $$
  and
  $$
    {
        \setlength\arraycolsep{3pt}
        \renewcommand\arraystretch{0.75}\bigl(\treetomat\circ\seqtotree\bigr)(x+x')=
        \begin{bmatrix}
          3     &       &   &       &       &   \\
          2     & \cdot &   &       &       &   \\
          \cdot & 2     & 2 &       &       &   \\
          \cdot & 1     & 1 & 4     &       &   \\
          \cdot & \cdot & 1 & 1     & \cdot &   \\
          \cdot & \cdot & 1 & \cdot & 2     & 1 \\
        \end{bmatrix}.
      }
  $$
\end{example}

\section{Final remarks}\label{section_final}

We have introduced Fishburn trees and Fishburn
covers, two classes of objects that transparently embody the combinatorial
structure of modified ascent sequences, Fishburn matrices and
$(\twoplustwo)$-free posets. The bijections relating these
families are illustrated in Figure~\ref{figure_maps_diagram}.
Fishburn trees act as a central hub from which every other Fishburn
structure can be easily derived.
Modified ascent sequences arise from the in-order traversal of
Fishburn trees; if a tree is drawn accordingly, then the sequence
is simply obtained by letting the labels of the tree fall under the
action of gravity (as in Figure~\ref{figure_fishburn_tree}).
In this sense, modified ascent sequences can be regarded as vertical
projections of Fishburn trees.
On the other hand, Fishburn matrices and $(\twoplustwo)$-free posets
stem from the $\rpath$-decomposition of Fishburn trees. Referring
once again to the usual representation of trees adopted in this paper,
matrices and posets can be seen as the projection of Fishburn trees
along their maximal right paths, that is, along the NW-SE branches.

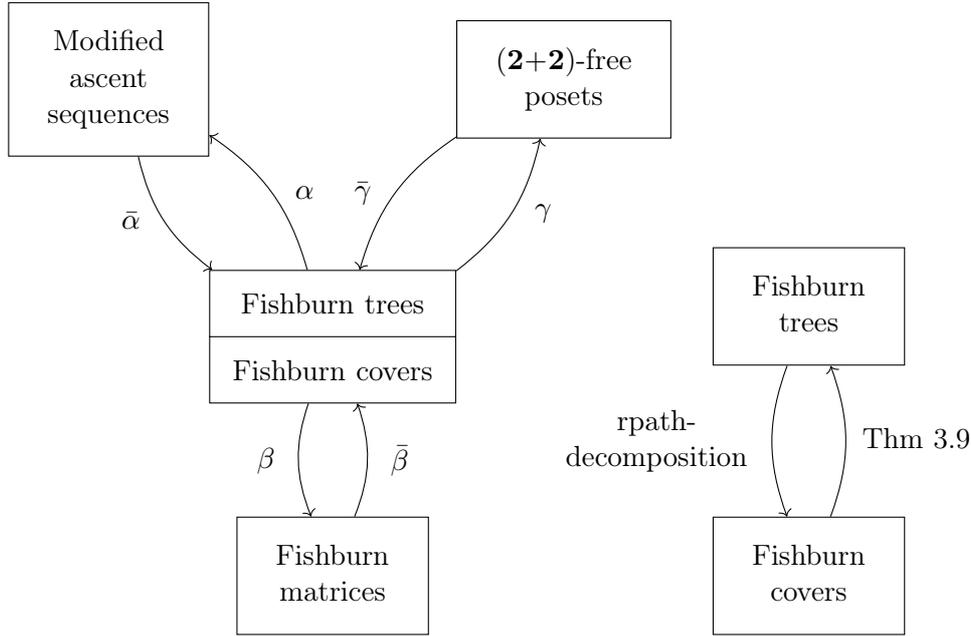
\begin{figure}
  \begin{center}
    $
      \begin{tikzpicture}[inner sep=3mm]
        \matrix [column sep=0mm, row sep=15mm]{
          \node[draw, shape=rectangle] (mod) {
            \begin{tabular}{c}
              Modified\\
              ascent\\
              sequences
            \end{tabular}
          }; &
          & \node[draw, shape=rectangle] (pos) {
            \begin{tabular}{c}
              $(\twoplustwo)$-free\\
              posets
            \end{tabular}}; \\
          & \node[draw, rectangle split, rectangle split parts=2] (tree) {%
            Fishburn trees\nodepart{second}Fishburn covers}; & \\
          & \node[draw, shape=rectangle] (mat) {%
            \begin{tabular}{c}
              Fishburn\\
              matrices
            \end{tabular}}; & \\
        };
        \draw (tree) edge[->,bend right=20] node[right]{$\treetoseq$} (mod);
        \draw (mod)  edge[->,bend right=20] node[left]{$\seqtotree$} (tree);
        \draw (tree) edge[->,bend right=20] node[right]{$\treetopos$} (pos);
        \draw (pos)  edge[->,bend right=20] node[left]{$\postotree$} (tree);
        \draw (tree) edge[->,bend right=20] node[left]{$\treetomat$}  (mat);
        \draw (mat)  edge[->,bend right=20] node[right]{$\mattotree$} (tree);
      \end{tikzpicture}\hspace{-22mm}
      \begin{tikzpicture}[inner sep=3mm]
        \matrix [row sep=20mm] {
          \node[draw, shape=rectangle] (tree) {
            \begin{tabular}{c}
              Fishburn\\ trees
            \end{tabular}}; \\
          \node[draw, shape=rectangle] (par) {
            \begin{tabular}{c}
              Fishburn\\
              covers
            \end{tabular}}; \\
        };
        \draw (tree) edge[->,bend right=20] node[left,xshift=2mm]{
          \begin{tabular}{c}
            $\rpath$-\\
            decomposition
          \end{tabular}} (par);
        \draw (par)  edge[->,bend right=20] node[right,xshift=-3mm]{
          \begin{tabular}{c}
            Thm \ref{pairs_to_tree}
          \end{tabular}} (tree);
      \end{tikzpicture}
    $
  \end{center}
  \caption{Bijections relating Fishburn trees and Fishburn covers to
    modified ascent sequences, Fishburn matrices and $(\twoplustwo)$-free
    posets.}\label{figure_maps_diagram}
\end{figure}

One may view Fishburn covers as encodings of the other
Fishburn structures in the following manner. Given a Fishburn cover
$P=B_1\dots B_k$, the other Fishburn objects are obtained by suitably
arranging the ``elementary blocks'' $B_1,\dots,B_k$:

\begin{itemize}
\item[$\fishtrees$:] To obtain a Fishburn tree~$T$, each elementary
  block $B_i$ is encoded as a maximal right path $W_i$. Diagonal paths
  form the tree $T_0$. Then the other paths are attached,
  one by one, to the leftmost occurrence of the corresponding
  label. This construction has been described in Theorem~\ref{pairs_to_tree}.

\item[$\Modasc$:] To obtain a modified ascent sequence~$x$, each elementary
  block $B_i$ is encoded as a decreasing sequence $\overrightarrow{B_i}$. Diagonal
  sequences are juxtaposed to obtain $x^{(0)}$. Then the remaining
  sequences are inserted one by one, each one immediately before the
  leftmost occurrence of the corresponding integer. This construction has
  been described in Section~\ref{section_flipsum}.

\item[$\fishmat$:] To obtain a Fishburn matrix~$A$, each elementary
  block $B_i$ is simply encoded as the $i$-th row of $A$ under the
  action of $\treetomat$.

\item[$\fishpos$:] To obtain a $(\twoplustwo)$-free poset~$Q$, each
  elementary block $B_i$ is encoded as the difference between
  two consecutive strict down-sets of $Q$, which are strictly ordered by
  inclusion, under the action of $\treetopos$.
\end{itemize}

In light of this, we could say that the Fishburn structures considered
here fall into two categories: Fishburn trees and modified sequences
are obtained by arranging their elementary blocks as dictated by the
leftmost occurrences of labels or integers. On the other hand, the most
trivial way of arranging elementary blocks---listing one block above the
other, as rows of a matrix or as strict down-sets of a poset---leads to
Fishburn matrices and $(\twoplustwo)$-free posets, respectively.

Fishburn permutations are related to modified ascent sequences
by the Burge transpose~\cite{BMCDK,CC}. The link between these two
structures has been extensively discussed~\cite{C2,CC}, and for this
reason we have decided to not include Fishburn permutations
in this paper. A deeper investigation on the relation
between Fishburn permutations and Fishburn trees is left
for future work.

As a first application of our framework, we have provided a
more direct solution to the $\flip$ and $\sumop$ problems
on modified ascent sequences.
A natural question is to investigate how the corresponding
operations act on Fishburn trees.
In Section~\ref{section_flip}, we showed that the $\flip$ operation
acts on the Fishburn cover of a Fishburn tree $T$ by mapping each
column $(i,j)$ to $(k+1-j,k+1-i)$, where $k=\max(T)$.
In other words, the $\flip$ of a Fishburn tree $T$ is obtained
by replacing the labels $\bigl(\blabel(v),\vlabel(v)\bigr)$
with $\bigl(k+1-\vlabel(v),k+1-\blabel(v)\bigr)$, for each
$v\in\vertset(T)$. A similar argument could be used to address
the $\sumop$ of two Fishburn trees.
Is there a more direct way of computing the $\flip$ and $\sumop$
of Fishburn trees? Also, is there any other natural
involution on the set of Fishburn trees, and how does the
corresponding operation act on modified ascent sequences,
Fishburn matrices and $(\twoplustwo)$-free posets?

Fishburn trees can be used to determine how
several statistics and subfamilies of Fishburn structures
are related to each other. We sketch some preliminary results below,
leaving a deeper investigation for future work.

A \emph{flat step} in a modified ascent sequence~$x$ is a pair of
consecutive entries $x_{i+1}=x_i$. Two elements of a poset~$Q$ are
\emph{indistinguishable} if they have the same down-set and up-set.
A modified ascent sequence is \emph{primitive} if it does not contain
flat steps and a poset is \emph{primitive} if it has no pairs of
indistinguishable elements. Furthermore, a modified ascent sequence
is \emph{self-modified} if it is equal to the corresponding (plain)
ascent sequence. Dukes and McNamara~\cite{DM} showed that self-modified ascent sequences, Fishburn matrices with positive diagonals, and
$(\twoplustwo)$-free posets with a chain of maximum length are all in bijection.

\begin{proposition}
  Let $T$ be a Fishburn tree. Let $x=\treetoseq(T)$, $A=\treetomat(T)$
  and $Q=\treetopos(T)$ be the corresponding modified ascent sequence,
  Fishburn matrix, and $(\twoplustwo)$-free poset, respectively. Then
  the following four conditions are equivalent:
  \begin{enumerate}
  \item $T$ is strictly-decreasing;
  \item $x$ is primitive;
  \item $A$ is binary;
  \item $Q$ is primitive.
  \end{enumerate}
  Similarly, the following four conditions are equivalent:
  \begin{enumerate}
  \item $T$ is comb-shaped;
  \item $x$ is self-modified;
  \item The main diagonal of $A$ is strictly positive;
  \item $Q$ contains a chain of maximum length.
  \end{enumerate}
\end{proposition}

We end with an open problem: Dukes and Parviainen~\cite{DP} described
the set of ascent sequences corresponding to bidiagonal Fishburn
matrices. What is the corresponding set of Fishburn trees?

\end{document}